\documentclass[10pt]{amsart}
\usepackage{graphicx,color,epsf}
\usepackage [cmtip,arrow]{xy}
\xyoption{all}

\usepackage {pb-diagram,pb-xy} 

\usepackage{amsmath,amssymb}
\newtheorem{theorem}{Theorem}[section]
\newtheorem{lemma}[theorem]{Lemma}
\newtheorem{corollary}[theorem]{Corollary}
\newtheorem{proposition}[theorem]{Proposition}

\theoremstyle{definition}
\newtheorem{definition}[theorem]{Definition}
\newtheorem{example}[theorem]{Example}

\newtheorem{notation}[theorem]{Notation}

\theoremstyle{remark}
\newtheorem{remark}[theorem]{Remark}

\definecolor{DarkBlue}{rgb}{0,0.1,0.55}
\newcommand{\bfdef}[1]{{\bf \color{DarkBlue}  \emph{#1}}}

\newcommand{\BIGOP}[1]{\mathop{\mathchoice%
{\raise-0.22em\hbox{\huge $#1$}}%
{\raise-0.05em\hbox{\Large $#1$}}{\hbox{\large $#1$}}{#1}}}
\newcommand{\bigtimes}{\BIGOP{\times}}

\numberwithin{equation}{section}

\newcommand {\hide}[1]{}

\newcommand {\junk}[1]{}

\newcommand {\R} {{\rm R}}

\newcommand {\K}     {\mbox{\rm K}}

\newcommand {\C}     {{\rm C}}

\newcommand {\Sphere}{\mbox{${\bf S}$}}     

\newcommand {\Z}  {\mathbb{Z}}
 
\newcommand {\Q}         {\mathbb{Q}}

\newcommand{\F}{\mathbb{F}}

\newcommand {\RR} {{\mathcal R}}

\newcommand {\V} {\mathbf{V}}

\newcommand {\la}   {{\langle}}
\newcommand {\ra}   {{\rangle}}

\newcommand {\PP}     {\mathbb{P}} 

\newcommand {\Rec}      {{\rm Rec}}
\newcommand {\Trunc}      {{\rm Trunc}}

\def\addots{\mathinner{\mkern1mu
\raise1pt\vbox{\kern7pt\hbox{.}}
\mkern2mu\raise4pt\hbox{.}\mkern2mu
\raise7pt\hbox{.}\mkern1mu}}

\newcommand{\HH}  {\mbox{\rm H}}

\newcommand{\defeq}{\;{\stackrel{\text{\tiny def}}{=}}\;}

\newcommand{\dlim}{\mathop{\lim\limits_{\longrightarrow}}}

\newcommand{\W}{\mathbf{W}}

\newcommand{\x}{\mathbf{x}}
\newcommand{\X}{\mathbf{X}}

\newcommand{\y}{\mathbf{y}}
\newcommand{\Y}{\mathbf{Y}}

\newcommand{\ZB}{\mathbf{Z}}

\newcommand{\Hom}{\mathrm{Hom}}

\begin{document}
\title[]
{A complex analogue of Toda's Theorem
}
\author{Saugata Basu}
\address{Department of Mathematics,
Purdue University, West Lafayette, IN 47906, U.S.A.}
\email{sbasu@math.purdue.edu}
\thanks{The 
author was supported in part by NSF 
grants CCF-0634907 and  CCF-0915954. 
Communicated by Peter B{\"u}rgisser.
} 

\date{\textbf{\today}}


\begin{abstract}
Toda \cite{Toda} proved in 1989 that the (discrete) polynomial time hierarchy,
$\mathbf{PH}$, 
is contained in the class $\mathbf{P}^{\#\mathbf{P}}$,
namely the class of languages that can be 
decided by a Turing machine in polynomial time given access to an
oracle with the power to compute a function in the counting complexity
class $\#\mathbf{P}$. This result, which illustrates the power of counting 
is considered to be a seminal result in computational complexity theory.
An analogous result (with a compactness hypothesis)
in the complexity theory over the reals  
(in the sense of  Blum-Shub-Smale real 
machines \cite{BSS89}) 
was proved in \cite{BZ09}. 
Unlike Toda's proof in the discrete case, which relied
on sophisticated combinatorial arguments, the proof in \cite{BZ09}
is topological in nature in which the properties of the topological join
is used in a fundamental way. However, the constructions used in 
\cite{BZ09} were semi-algebraic -- they used real inequalities in an
essential way and as such do not extend to the complex case. 
In this paper, we extend the techniques 
developed in \cite{BZ09} to the complex projective case. 
A key role is played by
the complex join of quasi-projective complex varieties. 
As a consequence we obtain a complex analogue of Toda's theorem. 
The results contained in this paper, taken together with those contained in 
\cite{BZ09}, illustrate the central role of the Poincar\'e polynomial 
in algorithmic algebraic geometry, as well as,
in computational complexity theory over the complex and real numbers
-- namely, the ability to compute it efficiently enables one to decide in 
polynomial time all languages in the 
(compact) polynomial hierarchy over the appropriate field.
\end{abstract}

\maketitle

\section{Introduction and Main Results}
\label{sec:intro}

\subsection{History and Background}
The primary motivation for this paper comes from 
classical (i.e. discrete) computational complexity theory. 
In classical complexity theory, there is a seminal result due to Toda
\cite{Toda}
linking the complexity of counting
with that of  deciding sentences with
a fixed number of quantifier alternations. 

More precisely, Toda's theorem gives the following inclusion 
(see Section \ref{subsec:counterparts} below or refer to \cite{Pap}
for precise definitions of the complexity classes appearing in the theorem).
\begin{theorem}[Toda \cite{Toda}]
\label{the:toda}
\[
{\bf PH} \subset {\bf P}^{\#{\bf P}}.
\]
\end{theorem} 
In other words, any language in the (discrete) polynomial hierarchy can
be decided by a Turing machine in polynomial time, given access to an
oracle with the power to compute  a function in $\#\mathbf{P}$.

\begin{remark}
The proof of Theorem \ref{the:toda} in \cite{Toda} 
is quite non-trivial.
While it is obvious that the classes $\mathbf{P},\mathbf{NP},\mathbf{coNP}$
are contained in ${\bf P}^{\#{\bf P}}$, the proof for the higher
levels of the polynomial hierarchy is quite intricate and proceeds in 
two steps: first proving that the $\mathbf{PH} \subset 
\mathbf{BP} \cdot\oplus\cdot \mathbf{P}$ (using previous  results of
Sch{\"o}ning \cite{Schoning},  and Valiant and Vazirani \cite{VV}), 
and then showing that
$\mathbf{BP} \cdot\oplus\cdot \mathbf{P} \subset {\bf P}^{\#{\bf P}}$.
Aside from the obvious question about what should be a proper analogue of
the complexity class $\# \mathbf{P}$ over the reals or complex numbers, 
because of the presence 
of complexity classes such as $\mathbf{BP}$ in the proof,
there seems to be no direct way of extending such a proof 
to real or complex complexity classes in the sense of 
Blum-Shub-Smale model of computation \cite{BSS89,Shub-Smale96}. 
This is not entirely surprising, since complexity results in the
Blum-Shub-Smale over different fields, while superficially similar,
often require completely different proof techniques. For example,
the fact that the polynomial hierarchy, $\mathbf{PH}$ 
is contained in the class $\mathbf{EXPTIME}$ is obvious over finite fields, 
but is non-trivial
to prove over real closed or algebraically closed fields (where
it is a consequence of efficient quantifier elimination algorithms).

The proof of the main theorem (Theorem \ref{the:main}) of this paper,
which can be seen as a complex analogue of Theorem \ref{the:toda}, proceeds
along completely different lines from the classical 
(that is over finite fields) case, and is mainly topological in nature.
\end{remark}

In the late eighties Blum, Shub, and Smale \cite{BSS89,Shub-Smale96}
introduced the notion
of Turing machines over more general fields, thereby generalizing
the classical problems of computational complexity theory such as
$\mathbf{P}$ vs. $\mathbf{NP}$ to corresponding problems over arbitrary fields
(such as the real, complex, $p$-adic numbers etc.) If one considers 
languages accepted by a  Blum-Shub-Smale machine 
over a finite field, one recovers the classical notions of discrete complexity
theory. Over the last two decades there has been a lot of research activity 
towards proving real as well as complex analogues of well known theorems
in discrete complexity theory. The first steps in this direction were taken by
the authors 
Blum, Shub, and Smale (henceforth B-S-S) 
themselves, when they proved the 
$\mathbf{NP}_{\C}$-completeness of the problem of deciding whether
a  systems of 
polynomial equations 
has a solution (in affine space)
(this is the complex analogue of Cook-Levin's theorem that the satisfiability
problem is $\mathbf{NP}$-complete in the discrete case), and subsequently 
through the work of several  researchers 
(Koiran, B{\"u}rgisser, Cucker, Meer to name 
a few) a well-established complexity theory over the reals as well as 
complex numbers have been built up, which mirrors closely the discrete case.

Indeed, one of the main attractions of the Blum-Shub-Smale computational model
is that it provides a framework to prove complexity results over more general
structures than just finite fields, 
with the hope that such results will help to
unravel the algebro-geometric underpinnings of the basic separation questions
amongst complexity classes. It is also often interesting to investigate complex
(as well as real) analogues of results in 
discrete complexity theory, because doing so reveals 
underlying geometric and topological phenomena not visible in the
discrete case. 
From this viewpoint
it is quite natural to seek complex (as well as real) analogues of 
Toda's theorem; and as we will see in this paper (see also \cite{BZ09}),
Toda's theorem properly interpreted
over the real and complex numbers 
gives an unexpected connection between two important but distinct 
strands of algorithmic algebraic
geometry -- namely, {\em decision problems} involving quantifier elimination on
one hand, and the problems of 
{\em computing topological invariants} of constructible sets on the
other.
Indeed, the original result of Toda, together with its real and
complex counter-parts seem to suggest a deeper connection of a model-theoretic
nature, between the
problems of efficient quantifier-elimination and efficient computation of 
certain 
discrete invariants of definable sets in a structure,
which might be an interesting problem on its own 
to explore further in the future.

\subsection{Recent Work}
There has been a large body of recent research on obtaining 
appropriate real (as well as complex) analogues of results in discrete
complexity theory, especially those related to counting complexity classes
(see \cite{Meer00,BC2,Burgisser-Cucker05,Burgisser-Cucker06}).
In \cite{BZ09} a real analogue of Toda's theorem was proved (with a compactness
hypothesis). In this paper we prove a similar result in the complex case.
Even though the basic approach is similar in both cases,
the topological tools in the complex case are different enough to merit a 
separate treatment. 
This is elaborated further in the next section 
(the main difficulty in extending the real arguments in \cite{BZ09}
to the complex case is 
that we can no longer use inequalities in our constructions).  

\subsection{Definitions of complexity classes}
In order to formulate
our result it is first necessary to define precisely 
complex counter-parts of the discrete 
polynomial time hierarchy ${\bf PH}$ and
the discrete complexity class $\# {\bf P}$,
and this is what we do next.

\subsubsection{Complex  counter-parts of ${\bf PH}$ and $\# {\bf P}$}
\label{subsec:counterparts}

For the rest of the paper $\C$ will denote an algebraically closed field of 
characteristic zero (there
is no essential loss in assuming that $\C = {\mathbb C}$) (indeed by a transfer
argument it suffices to prove all our results in this case).
By a \bfdef{complex machine} we will mean a  machine in the sense of 
Blum-Shub-Smale \cite{BSS89})
over the ground field $\C$.

\medskip
\noindent
{\em Notational convention.}
Since in what follows we will be forced to deal with multiple blocks of
variables in our formulas, we follow a notational convention by
which we denote blocks of variables by bold letters with superscripts 
(e.g. $\X^i$ denotes the $i$-th block), and we use non-bold letters
with subscripts to denote single variables (e.g. $X^i_j$ denotes the
$j$-th variable in the $i$-th block).
We use $\x^i$ to denote a specific value of the block of variables $\X^i$.

\begin{definition}
\label{def:multi-homogeneous}
We will call a quantifier-free first-order formula (in the language of fields),
$\phi(\X^1;\cdots;\X^\omega)$, having
several blocks of variables  $(\X^1,\ldots,\X^\omega)$ 
to be \bfdef{multi-homogeneous} if each polynomial appearing
in it is  multi-homogeneous in the blocks of variables 
$(\X^1,\ldots,\X^\omega)$
and such that $\phi$ is satisfied whenever any one of the blocks 
$\X^i = 0$.
Recall that a polynomial $P \in \C[\X^1;\cdots;\X^\omega]$ is multi-homogeneous
of multi-degree $(d_1,\ldots,d_\omega)$ if and only if it 
satisfies the identity
\[
P(\lambda_1 \X^1; \cdots ; \lambda_\omega \X^\omega) = 
\lambda_1^{d_1}\cdots \lambda_\omega^{d_\omega} P(\X^1;\cdots;\X^\omega).
\]
Clearly such a formula defines a \bfdef{constructible}
subset of $\PP^{k_1}_\C \times \cdots \times \PP^{k_\omega}_\C$ where the block
$\X^i$ is assumed to have $k_i+1$ variables. If $\omega = 1$, that is there is
only one block of variables, then we call $\phi$ a \bfdef{homogeneous} formula.
\end{definition}

\begin{notation}[Realization]
More generally, let 
$$
\displaylines{
\Phi(\X^1;\ldots;\X^\sigma) \defeq
(\mathrm{Q}_1 \Y^1) \cdots (\mathrm{Q}_\omega \Y^\omega) 
\phi(\X^1;\cdots;\X^\sigma;\Y^1;\cdots;\Y^\omega)
}
$$
be a (quantified) multi-homogeneous formula, with 
$\mathrm{Q}_i \in \{\exists,\forall\}, 1 \leq i \leq \omega$,
$\phi$ a quantifier-free multi-homogeneous formula,
and $\X^i$ (resp. $\Y^j$) is a block of $k_i + 1$ (resp. $\ell_j +1$)
variables.
We denote by $\RR(\Phi) \subset \PP_\C^{k_1} \times \cdots \times \PP_\C^{k_\sigma}$ 
the constructible set which is the \bfdef{realization} of the formula $\Phi$; 
i.e.,
$$ 
\displaylines{
\RR(\Phi(\X))=\{(\x^1,\ldots,\x^\sigma) \in \PP^{k_1}_\C\times\cdots\times
\PP_\C^{k_\sigma} \mid \cr 
(\mathrm{Q}_1 \y^1 \in \PP_\C^{\ell_1}) \cdots (\mathrm{Q}_\omega \y^\omega \in \PP_\C^{\ell_\omega}) 
\phi(\x^1;\cdots;\x^\sigma;\y^1;\cdots;\y^\omega)
\}.
}
$$
Sometimes, 
in order to emphasize the block structure in a multi-homogeneous
formula, we will write the quantifications as 
$(\exists \Y \in \PP_\C^\ell)$ (resp. $(\forall \Y \in \PP_\C^\ell)$) instead of
just $(\exists \Y)$ (resp. $(\forall \Y)$). This is purely notational
and does not affect the syntax of the formula. 
\end{notation}

\begin{notation}[Negation of a multi-homogeneous formula]
It is clear that the property of multi-homogeneity is preserved by the
Boolean operations of conjunction and disjunction. In order for it to
be preserved also under negation, we will adopt the
convention that the negation, $\neg \Phi(\X^1;\cdots;\X^\omega)$, of a
multi-homogeneous formula $\Phi(\X^1;\cdots;\X^\omega)$ is by definition
equal to
$$
\displaylines{
\tilde{\Phi} \vee \bigvee_{1 \leq i \leq \omega}(X^i = 0)
}
$$
where $\tilde{\Phi}$ is the usual negation of $\phi$ as a Boolean formula.
It is clear that defined this way, $\neg \Phi$ is multi-homogeneous, and 
\[
\RR(\neg{\Phi}) = \PP^{k_1}_\C \times \cdots \times \PP^{k_\omega}_\C
\setminus \RR(\Phi).
\]
\end{notation}

We say that two multi-homogeneous formulas, $\Phi$ and $\Psi$, are
\bfdef{equivalent} if $\RR(\Phi) = \RR(\Psi)$. Clearly, equivalent
multi-homogeneous formulas must have identical number of blocks of free 
variables,
and the corresponding block sizes must also be equal. 

Since the notion of multi-homogeneous formulas might look a bit unusual 
at first glance from the point of view of logic,
we illustrate below how to  homogenize
non-homogeneous formulas by considering the following simple example 
(which is a building block for the ``repeated squaring'' technique used to 
prove doubly exponential lower bounds for (real) quantifier elimination 
\cite{DH}).

\begin{example}
Let $\Phi(X)$ be the following (existentially) 
quantified non-homogeneous formula
expressing the fact, that $X^4 = 1$.

\begin{equation*}
\Phi(X) \defeq \exists Y (Y^2 -1 = 0) \wedge (Y - X^2 = 0).
\end{equation*}

A multi-homogeneous version of the same formula is given by:
\begin{equation*}
\Phi^h(X_0:X_1) 
\defeq \exists ((Y_0:Y_1) \in \PP^1_\C) (Y_1^2 -Y_0^2 = 0) \wedge (X_0^2Y_1 - X_1^2Y_0 = 0).
\end{equation*}

Notice that the quantifier-free bi-homogeneous formula 
\[
\Psi^h(X_0:X_1;Y_0:Y_1) \defeq (Y_1^2 -Y_0^2 = 0) 
\wedge (X_0^2Y_1 - X_1^2Y_0 = 0)
\]
defines a constructible subset of $\PP_\C^1 \times \PP_\C^1$, and
that the affine part of the constructible subset of $\PP_\C^1$ defined
by $\Phi^h$ coincides with the constructible subset of $\C^1$ 
defined by $\Phi(X)$. 
\end{example}

\subsection{Complex analogue of $\mathbf{PH}$}
The definition of the polynomial hierarchy over $\C$ mirrors
that of the discrete case (see \cite{Stockmeyer}) very closely.

\begin{definition}[The class $\mathbf{P}_{\C}$]
A sequence 
$$
\displaylines{
\left(T_n \subset \C^{n}
\right)_{n > 0}
}
$$  
of constructible subsets
is said to belong to the class $\mathbf{P}_{\C}$
if there exists a 
B-S-S machine $M$ over $\C$ 
(see~\cite{BSS89, BCSS98}),
such that for all $\x \in \C^{n}$, the machine $M$ 
decides  membership of $\x$ in $T_n$ 
in time bounded by a polynomial in $n$.

More generally,
suppose that $k(n)$ is some fixed polynomial which is non-negative and 
increasing.  Let
$(T_n \subset \C^{k(n)})_{n >0}$ be a sequence of constructible sets.
We will say that $(T_n \subset \C^{k(n)})_{n >0}$ belongs to
$\mathbf{P}_{\C}$ if the sequence $(S_n \subset \C^{n})_{n > 0}$ belongs to
$\mathbf{P}_{\C}$, where $S_n$ is defined by 
$$
\displaylines{
S_{k(n)} = T_{k(n)}, \mbox{ for all } n>0, \cr
S_m = \emptyset, \mbox{ otherwise}.
}
$$ 
\end{definition}

\begin{definition}[The classes ${\bf \Sigma}_{\C,\omega}$ and 
${\bf \Pi}_{\C,\omega}$]
\label{def:sigma}
Let $\omega \geq  0$ be a fixed integer.
A sequence 
$$
\displaylines{
\left(
S_n \subset \C^{n}
\right)_{n > 0}
}
$$ 
of constructible subsets
is said to be in 
the complexity class ${\bf \Sigma}_{\C,\omega}$, 
if for each $n > 0$, the constructible set
$S_n$ is described by a first order formula
\begin{equation}\label{eq:alternation}
(\mathrm{Q}_1 \Y^{1} )  \cdots (\mathrm{Q}_\omega \Y^{\omega} )
\phi_n(X_1,\ldots,X_{n},\Y^1,\ldots,\Y^\omega),
\end{equation}
with $\phi_n$ a quantifier free formula in the first order theory of $\C$,
and for each $i, 1 \leq i \leq \omega$,
$\Y^i = (Y^i_1,\ldots,Y^i_{n})$ is a block of $n$ variables,
$\mathrm{Q}_i \in \{\exists,\forall\}$, with $\mathrm{Q}_j \neq \mathrm{Q}_{j+1}, 1 \leq j < \omega$,
$\mathrm{Q}_1 = \exists$,
and 
the sequence 
$$\displaylines{
\left(
T_n \subset \C^{n}\times \underbrace{\C^{n} \times
 \cdots \times \C^{n}}_{\omega \mbox{ times}}\right)_{n >0}
}
$$
of constructible subsets   
defined by the quantifier-free formulas $(\phi_n)_{n>0}$ 
belongs to the class ${\bf P}_{\C}$.

Similarly, the complexity class ${\bf \Pi}_{\C,\omega}$
is defined as in Definition~\ref{def:sigma}, with the difference 
that the  alternating quantifiers in~\eqref{eq:alternation} start with 
$\mathrm{Q}_1=\forall$.
\end{definition}

\begin{remark}
\label{rem:padding}
Notice that in Definition \ref{def:sigma} there is no loss of generality
in assuming that the sizes of the blocks of variables
$\X,\Y^1,\ldots,\Y^\omega$ are all equal. 
To be more precise, suppose
that the size of block $\X$ is $k(n)$, and that  of $\Y^i$ is $k_i(n)$ for 
$1 \leq i \leq \omega$, where 
$k(n),k_1(n),\ldots,k_\omega(n)$ are fixed non-negative polynomials.
Let 
$$
\displaylines{
\left(
S_n \subset \C^{k(n)}
\right)_{n > 0}
}
$$ 
be a sequence of constructible subsets
described by a first order formula
\begin{equation}
(\mathrm{Q}_1 \Y^{1} )  \cdots (\mathrm{Q}_\omega \Y^{\omega} )
\phi_n(\X,\Y^1,\ldots,\Y^\omega),
\end{equation}
with $\phi_n$ a quantifier free formula in the first order theory of $\C$,
and
$\mathrm{Q}_i \in \{\exists,\forall\}$, with $\mathrm{Q}_j \neq \mathrm{Q}_{j+1}, 1 \leq j < \omega$,
such that
the sequence 
$$\displaylines{
\left(
T_n \subset \C^{k(n)}\times \C^{k_1(n)} \times
 \cdots \times \C^{k_\omega(n)}\right)_{n >0}
}
$$
of constructible subsets   
defined by the quantifier-free formulas $(\phi_n)_{n>0}$ belongs to
$\mathbf{P}_\C$.

Let $\tilde{k}(n)$ be any non-negative polynomial which 
majorizes $k(n),k_1(n),\ldots,k_\omega(n)$,
and let 
$\tilde{\X} = (\X,\X'), \tilde{\Y}^i = (\Y^i,{\Y^i}'), 1 \leq i \leq \omega$,
be blocks of variables obtained from the blocks $\X,\Y^i$,
of size $\tilde{k}(n)$ by padding by an appropriate
number of extra variables, $\X',{\Y^i}'$, respectively.
By identifying the subspace of $\C^{\tilde{k}(n)}$
defined by setting the variables in the block $\X'$ to $0$, 
with $\C^{k(n)}$ (and thus identifying $S_n$ with its image
under the corresponding inclusion in $\C^{\tilde{k}(n)}$),
we have a sequence
$$
\displaylines{
\left(
S_n \subset \C^{\tilde{k}(n)}
\right)_{n > 0}
}
$$ 
of constructible subsets described by the formula
\begin{equation}
(\mathrm{Q}_1 \tilde{\Y}^{1})  \cdots (\mathrm{Q}_\omega \tilde{\Y}^{\omega})
\phi_n({\X},{\Y}^1,\ldots,{\Y}^\omega) \wedge (\X' = 0).
\end{equation}

It is clear that the sequence
$$\displaylines{
\left(
T_n \subset \C^{k(n)}\times \C^{k_1(n)} \times
 \cdots \times \C^{k_\omega(n)}\right)_{n >0}
}
$$
belongs to the class $\mathbf{P}_\C$ if and only if the sequence 
$$\displaylines{
\left(
\tilde{T}_n \subset \C^{\tilde{k}(n)}\times \C^{\tilde{k}(n)} \times
 \cdots \times \C^{\tilde{k}(n)}\right)_{n >0}
}
$$
defined by 
\[
\tilde{\phi}_n(\tilde{\X},\tilde{\Y}^1,\ldots,\tilde{\Y}^\omega)
:= 
\phi_n({\X},{\Y}^1,\ldots,{\Y}^\omega) \wedge (\X' = 0).
\]
belongs to the class $\mathbf{P}_\C$. 
In other words (up to padding by some additional variables $\X'$ as above)
there is no loss of generality in assuming that all the block sizes
are equal. 
\end{remark}

Since, adding an additional block of quantifiers on the outside
(with new variables 
that do not appear in the quantifier-free formula $\phi_n$) 
does not change the set defined by a quantified formula
we have the following inclusions:
$$ {\bf \Sigma}_{\C,\omega}\subset {\bf \Pi}_{\C,\omega+1}, 
\text{\ and\ }
{\bf \Pi}_{\C,\omega}\subset{\bf \Sigma}_{\C,\omega+1}.
$$

Note that by the above definition the class 
${\bf \Sigma}_{\C,0} = {\bf \Pi}_{\C,0}$ 
is the 
class ${\bf P}_{\C}$, 
the class ${\bf \Sigma}_{\C,1} = {\bf NP}_{\C}$ and the
class ${\bf \Pi}_{\C,1} = {\bf \mbox{co-}NP}_{\C}$.

\begin{definition}[Complex polynomial hierarchy] 
The complex polynomial time hierarchy is defined to be the union
\[
{\bf PH}_{\C} \defeq \bigcup_{\omega \geq 0} 
({\bf \Sigma}_{\C,\omega} \cup {\bf \Pi}_{\C,\omega}) = 
\bigcup_{\omega \geq 0} {\bf \Sigma}_{\C,\omega}  = 
\bigcup_{\omega \geq 0} {\bf \Pi}_{\C,\omega}.
\]
\end{definition}

As in the real case studied in \cite{BZ09}
for technical reasons 
we need to restrict to compact constructible sets.
However, unlike in \cite{BZ09} where the compact languages consisted
of closed semi-algebraic subsets of spheres, in this paper we consider
closed subsets of projective spaces instead. This is a much more natural
choice for defining compact complex complexity classes. 

We now define the compact analogue of
${\bf PH}_{\C}$ that we will denote ${\bf PH}_{\C}^c$.
Unlike in the non-compact case, we will assume all variables
vary over certain compact sets (namely complex projective spaces 
of varying dimensions).

We first need to be precise about what we mean by a complexity class
of sequences of constructible subsets of complex projective spaces.

\begin{notation}[Affine cones]
\label{not:affinecone}
For any constructible subset $S \subset \PP^k_\C$ we denote by 
$C(S) \subset \C^{k+1}$ the affine cone over $S$.
More generally, if $S \subset \PP_\C^{k_1} \times \cdots \times \PP_\C^{k_\omega}$
is a constructible subset, then $C(S)\subset \C^{k_1+1} \times \cdots \times
\C^{k_\omega+1}$ will denote the union of 
$L^1 \times \cdots \times L^\omega$ such that each $L^i \subset \C^{k_i+1}$ is 
a line through the origin, such that the point in 
$\PP_\C^{k_1} \times \cdots \times \PP_\C^{k_\omega}$ represented by 
$(L^1,\ldots,L^\omega)$
is in $S$.
\end{notation}

\begin{definition}
\label{def:compactP}
We say that a sequence 
$$\displaylines{
\left(
S_n \subset 
\underbrace{\PP^n_\C \times \cdots \times \PP^n_\C}_{n \mbox{ times}}
\right)_{n > 0}
}
$$ 
of constructible subsets is in the complexity class $\mathbf{P}_{\C}$, 
if the sequence of affine cones
$\left(C(S_n) \subset 
\underbrace{\C^{n+1} \times \cdots \times \C^{n+1}}_{n \mbox{ times}}
\right)_{n>0}
$
belongs to the complexity class $\mathbf{P}_\C$.
\end{definition}

\begin{remark}
The subspaces spanned by the increasing sequence of standard  basis
elements of 
\[
\C = \la e_0 \ra \subset \C^2 = \la e_0,e_1\ra \subset \cdots \subset \C^{n+1} = \la e_0,\ldots,e_n \ra \subset \cdots
\]
after projectivization gives a flag 
\[
\PP^0_\C \subset \PP^1_\C \subset \cdots \subset \PP^{n}_\C \subset \cdots
\]
For $0 \leq m\leq n$, let $\iota_{m,n}: \PP^{m}_\C \hookrightarrow \PP^n_\C$
denote the corresponding inclusion.

Now, if 
$(S_n \subset \PP^n_\C)_{n > 0}$ is a sequence 
of constructible sets, we can after identifying
$\PP^n_\C$ with the subspace 
\[
\PP^{n}_\C \times
\underbrace{\iota_{0,n}(\PP^0_\C) \times \cdots \times \iota_{0,n}(\PP^0_\C)}_{n-1 \mbox{ times}}
\]
of
$\underbrace{\PP^n_\C \times \cdots \times \PP^n_\C}_{n \mbox{ times}}$
identify the sequence $(S_n \subset \PP^n_\C)_{n > 0}$ with 
the sequence 
$$
\displaylines{
(\tilde{S_n} \subset \underbrace{\PP^n_\C \times \cdots \times \PP^n_\C}_{n \mbox{ times}})_{n > 0}.
}
$$
where
$$
\tilde{S_n} = S_n \times
\underbrace{\iota_{0,n}(\PP^0_\C) \times \cdots \times \iota_{0,n}(\PP^0_\C)}_{n-1 \mbox{ times}}.
$$

We will (by abuse of language)  
say that the sequence $(S_n \subset \PP^n_{\C})_{n > 0}$ belongs to the
class $\mathbf{P}_\C$ if the sequence 
$$
\displaylines{
(\tilde{S_n} \subset \underbrace{\PP^n_\C \times \cdots \times \PP^n_\C}_{n \mbox{ times}})_{n > 0}.
}
$$
belongs to class $\mathbf{P}_\C$.

More generally, suppose that $m(n)$ is a non-negative polynomial in $n$ and,
$(k_i(n))_{i > 0}$ a sequence of non-negative polynomials such that there
exists a polynomial $k(n)$ which majorizes
$m(n),k_1(n),k_2(n),\ldots,k_{m(n)}(n)$
for all $n>0$. For example, we could have $m(n) = n$, and 
$k_i(n) = i n$. Clearly, in this case the polynomial $k(n) = n^2+1$
majorizes 
$m(n),k_1(n),k_2(n),\ldots,k_{m(n)}(n)$ are 
for all $n>0$.

We say that a sequence 
$$
\displaylines{
\left(
S_n \subset \PP_{\C}^{k_1(n)} \times \cdots \times \PP_{\C}^{k_{m(n)}(n)} \right)_{n > 0}
}
$$ 
is in $\mathbf{P}_\C$,
if the sequence
$$\displaylines{
\left(
T_n \subset 
\underbrace{\PP^n_\C \times \cdots \times \PP^n_\C}_{n \mbox{ times}}
\right)_{n > 0}
}
$$ 
is in $\mathbf{P}_\C$, where
$$
\displaylines{
T_{k(n)} = \tilde{S_n} \mbox{ for all } n> 0, \cr
T_n = \emptyset, \mbox{ otherwise }
}
$$
and
$$
\tilde{S}_n \subset 
\underbrace{\PP_{\C}^{k(n)} \times \cdots \times \PP_{\C}^{k(n)}}_{ k(n) \mbox{ times}}
$$
is defined by 
$$
\displaylines{
\tilde{S}_n = 
\iota_{k_1(n),k(n)} \times \cdots \times \iota_{k_{m(n)},k(n)}(S_n)
\times 
\underbrace{\iota_{0,k(n)}(\PP^0_\C) \times \cdots \times 
\iota_{0,k(n)}(\PP^0_\C)}_{k(n) - m(n)
\mbox{ times }}
.
}
$$
\end{remark}

\begin{definition}[Compact projective 
version of ${{\mathbf \Sigma}}_{\C,\omega}$]
\label{def:compactpolynomialhierarchy}
We say that a sequence 
$$\displaylines{
\left(
S_n \subset \underbrace{\PP_{\C}^{n} \times \cdots \times \PP_{\C}^{n}}_{n \mbox{ times}} \right)_{n > 0}
}
$$ 
of constructible subsets 
is in the complexity class ${\bf \Sigma}_{\C,\omega}^c$,
if for each $n > 0$,
$S_n$ is described by a first order formula
\[
 (\mathrm{Q}_1 \Y^{1} \in \PP_{\C}^{n})  
\cdots (\mathrm{Q}_\omega \Y^{\omega} \in 
\PP_{\C}^{n} )
\phi_n(\X^1;\cdots;\X^{n};\Y^1;\cdots;\Y^\omega),
\]
with $\phi_n$ a quantifier-free first order multi-homogeneous
formula defining a 
{\em closed} (in the Zariski topology) subset of 
\[
\underbrace{\PP_{\C}^{n} \times \cdots \times 
\PP_\C^{n}}_{n \mbox{ times}} \times  
\underbrace{\PP_{\C}^{n}\times\cdots\times
\PP_{\C}^{n}
}_{\omega \mbox{ times }},
\]
$\mathrm{Q}_i \in \{\exists,\forall\}$, 
$\mathrm{Q}_1 = \exists$, and
the sequence of constructible sets  $(T_n)_{n >0}$ 
defined by the formulas $(\phi_n)_{n >0}$
belongs to the class $\mathbf{P}_{\C}$.
\end{definition}

\begin{remark}
As remarked before (cf. Remark \ref{rem:padding}), 
it is not essential to have all the block
sizes to be equal in the above definition as long as all
the number and the sizes of the blocks are polynomially bounded, 
and we will by 
a slight abuse of language allow polynomially bounded
number of blocks with polynomially bounded, but not necessarily
equal,  block sizes in 
what follows without further remark.
\end{remark}


\begin{example}
\label{eg:compact}
We give a very natural example of a language 
in $\mathbf{\Sigma}_{\C,1}^c$ (i.e. the 
compact version of $\mathbf{NP}_{\C}$).
Let $k(n,d) = \binom{n+d}{d}$ and identify 
\[
\underbrace{\PP_{\C}^{k(n,d)-1} \times \cdots \times \PP_{\C}^{k(n,d)-1}}
_{n+1 \mbox{ times}}
\]
with systems of $n+1$ homogeneous polynomials 
in 
$n+1$
variables of degree 
$d$. 
Let 
\[
S_{n,d} \subset \underbrace{\PP_{\C}^{k(n,d)-1} \times \cdots \times \PP_{\C}^{k(n,d)-1}}
_{n+1 \mbox{ times}}
\]
be 
defined by 
$$
\displaylines{
S_{n,d} = \{(P_1;\cdots;P_{n+1}) \;\mid\;  P_i \in
\PP_{\C}^{k(n,d)-1} \mbox{ and } \exists~\x 
=(x_0:\cdots:x_n) \in \PP_{\C}^n \mbox{ with } \cr
P_1(\x) = \cdots = P_{n+1}(\x) = 0 \}.
}
$$
In other words,  $S_{n,d}$ is the set of systems of $(n+1)$ 
homogeneous polynomial equations of degree  $d$, 
which have a zero in $\mathbb{P}^n_{\C}$.
Then it is clear from the definition of the class $\mathbf{\Sigma}_{\C,1}^c$ 
that for any fixed $d > 0$,
$$
\displaylines{
\left(
S_{n,d} \subset \underbrace{\PP_{\C}^{k(n,d)-1} \times \cdots \times \PP_{\C}^{k(n,d)-1}}
_{n+1 \mbox{ times}}
\right)_{n > 0} \in \mathbf{\Sigma}_{\C,1}^c.
}
$$

Note that it is \emph{not known} if for any fixed $d$ 

$$
\displaylines{
\left(
S_{n,d} \subset \underbrace{\PP_{\C}^{k(n,d)-1} \times \cdots \times \PP_{\C}^{k(n,d)-1}}
_{n+1 \mbox{ times}}
\right)_{n > 0}
}
$$
is $\mathbf{NP}_\C$-complete,
while the non-compact version of this language
i.e. the language consisting of
systems of polynomials having a 
zero in $\C^n$ (instead of $\mathbb{P}^n_\C$),  
has been shown to be $\mathbf{NP}_\C$-complete for $d \geq 2$ \cite{BCSS98}.
\end{example}

We define analogously the class ${\bf \Pi}_{\C,\omega}^c$, and finally 
define:
\begin{definition}
The \bfdef{compact projective polynomial hierarchy} over $\C$ 
is defined to be the union
\[
{\bf PH}_{\C}^c \defeq 
\bigcup_{\omega \geq 0} ({\bf \Sigma}_{\C,\omega}^c \cup 
{\bf \Pi}_{\C,\omega}^c) = \bigcup_{\omega \geq 0} {\bf \Sigma}_{\C,\omega}^c =
\bigcup_{\omega \geq 0} {\bf \Pi}_{\C,\omega}^c.
\]
\end{definition}

Notice that the constructible subsets belonging to any language in 
${\bf PH}_{\C}^c$ are all compact 
(in fact Zariski closed subsets of complex projective spaces). 

\begin{remark}
\label{rem:compact_classes}
The compact classes introduced above might be of interest in their own right.
As remarked earlier it is not known whether the compact language
$$
\displaylines{
\left(
S_{n,d} \subset \underbrace{\PP_{\C}^{k(n,d)-1} \times \cdots \times \PP_{\C}^{k(n,d)-1}}
_{n+1 \mbox{ times}}
\right)_{n > 0}
}
$$
in Example \ref{eg:compact} 
is $\mathbf{NP}_\C$-complete. It is important to resolve this
question in order to understand whether the hardness of solving 
polynomial systems over $\C$ is due to the non-compactness of the (affine) 
solution space, or 
due to some intrinsic algebraic reasons. 
\end{remark}

\subsubsection{Complex projective analogue of $\#{\bf P}$} 
\label{sec:sharpP}
We now define the complex analogue of $\#{\bf P}$
(cf. the class $\#{\bf P}_\R^{\dagger}$ defined in \cite{BZ09} in the real 
case).

We first need a notation.

\begin{notation}[Poincar\'e polynomial]
In case $\C = \mathbb{C}$,
for any constructible subset $S\subset \PP_\C^k$ 
we denote by $b_i(S)$ the $i$-th
Betti number (that is the rank of the singular 
homology group $\HH_i(S) \defeq \HH_i(S,\Q)$) of $S$. 

We also let $P_S \in \Z[T]$ denote the \bfdef{Poincar\'e polynomial} of $S$, 
namely

\begin{equation}
\label{def:Poincarepolynomial}
P_S(T)\; \defeq\; \sum_{i \geq 0}  b_i(S)\; T^i.
\end{equation}
\end{notation}

\begin{remark}
\label{rem:coefficients}
Since we are only going to be concerned with the Betti numbers of constructible
sets, we do not lose any information by considering homology groups
with coefficients in $\Q$ rather than in $\Z$, noting that
\[
\HH_i(S,\Q) = \HH_i(S,\Z) \otimes_\Z \Q.
\]
Note also that in this case by the universal coefficient theorem for cohomology
\cite{Spanier}, we have  that the cohomology groups 
\[
\HH^i(S) \defeq \HH^i(S,\Q) \cong \Hom(\HH_i(S,\Q),\Q).
\]
\end{remark}

\begin{remark}
Over an arbitrary algebraically closed field $\C$ of characteristic $0$, 
ordinary  singular homology is not well defined. 
We use a modified homology theory 
(which agrees with singular homology in case $\C = {\mathbb C}$
and which is homotopy invariant)  as done in \cite{BPRbook2} in case
of semi-algebraic sets over arbitrary real closed fields (see 
\cite{BPRbook2}, page 279).
Note that
by taking real and imaginary parts, 
every constructible set over $\C$ is a semi-algebraic 
set over an appropriate real closed subfield
-- namely, a maximal real subfield of $\C$.
\end{remark}

For the rest of the paper we will assume $\C = \mathbb{C}$, noting that all
the results generalize to arbitrary algebraically closed fields of 
characteristic $0$ using the transfer principle.

\begin{definition}[The class $\#{\bf P}_{\C}^{\dagger}$] 
\label{def:sharp}
We say a sequence of constructible functions 
$$
\displaylines{
\left(f_n:\PP_\C^n \rightarrow \Z[T]
\right)_{n > 0}
}
$$
is in the class $\#\mathbf{P}_{\C}^{\dagger}$,
if there exists a sequence
$$
\displaylines{
\left(
S_n \subset 
\PP_\C^n \times 
\underbrace{\PP_{\C}^{n} \times \cdots \times \PP_{\C}^{n}}
_{n \mbox{ times }}\right)_{n > 0}
\in \mathbf{P}_\C,
}
$$
such that
\[
f_{n}(\x) = P_{S_n,\x} 
\]
for each $\x \in \PP_\C^n$, where 
$S_{n,\x} = S_{n} \cap \pi_n^{-1}(\x)$ 
and 
\[
\pi_n: \PP_\C^n \times 
\underbrace{\PP_{\C}^{n} \times \cdots \times \PP_{\C}^{n}}_{n \mbox{ times }}
\rightarrow \PP_\C^n
\] 
is the projection along the last co-ordinates. 
\end{definition}

\begin{remark}
We make a few remarks about the class $\#{\bf P}_{\C}^{\dagger}$ 
defined above. First of all notice that
the  class $\#{\bf P}_{\C}^{\dagger}$ is quite robust.
For instance, given two sequences 
$(f_n)_{n > 0}, (g_n)_{n > 0} \in \#{\bf P}_{\C}^{\dagger}$ it follows
(by taking disjoint union of the corresponding constructible sets) that 
$(f_n+ g_n)_{n > 0} \in \#{\bf P}_{\C}^{\dagger}$, and also 
$(f_n g_n)_{n > 0} \in \#{\bf P}_{\C}^{\dagger}$ 
(by taking Cartesian product of the corresponding constructible sets
and using the multiplicative
property of the Poincar\'e polynomials, which itself is a consequence
of the Kunneth formula in homology theory \cite{Spanier}.)
\end{remark}

\begin{remark}
\label{rem:zeta}
The connection between counting points of varieties and their Betti numbers is 
more direct over fields of positive characteristic via the zeta function.
The zeta function of a variety defined over $\F_p$ is the exponential 
generating function of the sequence whose $n$-th term is the number of
points in the variety over $\F_{p^n}$. The zeta function of such a variety
turns out to be a rational function in one variable 
(a deep theorem of algebraic geometry first conjectured by Andre Weil
\cite{Weil}
and proved by Dwork \cite{Dwork} and  Deligne \cite{Deligne1, Deligne2}), 
and its numerator and denominator are products of 
polynomials whose degrees are the  Betti numbers of the variety 
with respect to a certain ($\ell$-adic) co-homology theory. The point of this 
remark is that the problems  of  ``counting'' varieties and computing their 
Betti numbers, are connected at a deeper level, and thus our 
choice of definition for a complex analogue 
of $\#{\bf P}$ is not altogether ad hoc.
\end{remark}

\begin{remark}
\label{rem:justification}
A different definition of the class $\#\mathbf{P}_\C^{\dagger}$ (more in line with 
previous work of B{\"u}rgisser  et al. 
\cite{Burgisser-Cucker05}) would be obtained by replacing
in Definition \ref{def:sharp} the Poincar\'e polynomial, $P_S(T)$, by the 
Euler-Poincar\'e characteristic i.e. the value of $P_S$ at $T = -1$.
The Euler-Poincar\'e characteristic is additive (at least 
when restricted to
complex varieties), and thus has some attributes of being a discrete analogue
of volume. But at the same time it should be noted that the 
Euler-Poincar\'e characteristic is a rather weak invariant -- for instance,
it does not determine the number of connected components
of a given variety. Also notice that in the case of finite fields referred to
in Remark \ref{rem:zeta}, all the Betti numbers, not just their alternating 
sum, enter (as degrees of factors) 
in the rational expression for the zeta function of a variety.
While it would certainly be a much stronger reduction result if one
could obtain a Toda-type theorem using only the Euler-Poincar\'e characteristic
instead of the whole Poincar\'e polynomial, it is at present unclear 
if such a theorem can be proven (see also Section \ref{sec:future}(C)).  
\end{remark}

\section{Statements of the main theorems}
We can now state the main result of this paper.
\begin{theorem}[Complex analogue of Toda's theorem]
\label{the:main}
\[
{\bf PH}^c_{\C} \subset {\bf P}_{\C}^{\#{\bf P}_{\C}^{\dagger}}.
\]
\end{theorem}

\begin{remark}
Note that following the usual convention 
${\bf P}_{\C}^{\#{\bf P}_{\C}^{\dagger}}$ denotes the class of 
languages accepted by a B-S-S machine over $\C$ in polynomial time
with access to an oracle which can compute functions in 
$\#{\bf P}_{\C}^{\dagger}$.
\end{remark}

\begin{remark}
We leave it as an open problem to prove Theorem \ref{the:main} with
${\bf PH}_{\C}$ instead of ${\bf PH}^c_{\C}$ on the left hand side. 
However, we also note that many theorems of complex algebraic  
geometry take their
most satisfactory form in the case of complete varieties, which is the setting
considered in this paper.
\end{remark}

As a consequence of our method, we obtain a reduction
(Theorem~\ref{the:main2})
that might be of independent interest.
We first define the following two problems:

\begin{definition}[Compact general decision problem with at most 
$\omega$ quantifier alternations (${\bf GDP_{\C,\omega}^c}$)]
The input and output for this problem are as follows.
\begin{itemize}
\item{\bf{Input.}} 
A sentence $\Phi$
\[    (\mathrm{Q}_1 \X^{1} \in \PP^{k_1}_\C)  \cdots 
(\mathrm{Q}_\omega \X^{\omega} \in \PP_\C^{k_\omega})
\phi(\X^1;\ldots;\X^\omega),
\]
where for each $i, 1 \leq i \leq \omega$,
$\mathrm{Q}_i \in \{\exists,\forall\}$, with $\mathrm{Q}_j \neq \mathrm{Q}_{j+1}, 1 \leq j < \omega$,
and
$\phi$ is a quantifier-free multi-homogeneous formula 
defining a {\em closed}
subset $S$ of $\PP^{k_1}\times\cdots\times\PP^{k_\omega}$. 
\item{\bf{Output.}} 
True or False depending on whether $\Phi$ is true or false.
\end{itemize}
\end{definition}

\begin{definition}[Computing the Poincar\'e polynomial of 
constructible  sets (\textbf{{\em Poincar\'e}})]
The input and output for this problem are as follows.
\begin{itemize}
\item{\bf{Input.}} A quantifier-free homogeneous formula defining a 
constructible  subset $S\subset \PP_\C^k$.
\item{\bf{Output.}} 
The Poincar\'e polynomial $P_S(T)$.
\end{itemize}
\end{definition}

\begin{theorem}
\label{the:main2}
For every  $\omega > 0$, there is a deterministic 
polynomial time reduction in the 
Blum-Shub-Smale model of 
$\bf{GDP_{\C,\omega}^c}$ to \textbf{Poincar\'e}.
\end{theorem}

\begin{remark}
We remark that (in contrast to the real case)
in the complex case, we are able to prove a slightly stronger result
than stated above in Theorem \ref{the:main2}.
Our proof of  Theorem \ref{the:main2} gives a polynomial time
reduction of $\bf{GDP_{\C,\omega}^c}$ to the problem of computing the 
{\em pseudo-Poincar\'e} polynomial 
(defined below, see Eqn. \ref{def:pseudo-Poincarepolynomial})
of constructible sets. The pseudo-Poincar\'e polynomial is easily computable
from the Poincar\'e polynomial.
\end{remark}

\subsection{Outline of the main ideas and contributions}
The basic idea behind the proof of a real analogue of Toda's theorem
in \cite{BZ09}
is a topological construction, which 
given a semi-algebraic set $X \subset \R^{m} \times \R^{n}$,  $p \geq 0$, 
and $\mathrm{pr}_1: \R^{m} \times \R^n \subset \R^n$ the projection
on $\R^m$
constructs
{\em efficiently} a semi-algebraic set, $D^p(X)$, such that 
\begin{equation}\label{eq:D}
b_i(\mathrm{pr}_1(X)) = b_i(D^p(X)), \quad 0 \leq i < p.
\end{equation}
Moreover, membership in $D^p(X)$ can be tested efficiently if the same is
true for $X$. Note that this last property will not hold in general for the set
$\mathrm{pr}_1(X)$ itself (unless of course $\mathbf{P}_\R = \mathbf{NP}_\R$).

The topological construction used in the definition of $D^p(X)$
in \cite{BZ09} is the iterated fibered join,
$J^p_{\mathrm{pr}_1}(X)$, of a 
semi-algebraic set $X$ with itself over a projection map $\mathrm{pr}_1$. 
There is also an induced  surjective map $J^p_{\mathrm{pr}_1}(X)\rightarrow \mathrm{pr}_1(X)$
which we denote by $\mathrm{pr}_1^{(p)}$. 
The fibers of
this induced map $\mathrm{pr}_1^{(p)}: J^p_{\mathrm{pr}_1}(X) \rightarrow \mathrm{pr}_1(X)$, 
over a point 
$\x \in \mathrm{pr}_1(X)$,
are then ordinary  $(p+1)$-fold joins of the fiber 
$(\mathrm{pr}_1^{(p)})^{-1}(\x)$, and by 
connectivity properties of the join are $p$-connected. It is now possible
to prove
using a version of the 
Vietoris-Beagle theorem that the map $\mathrm{pr}_1^{(p)}$ is a 
$p$-equivalence (see \cite{BZ09} for the precise definition
of $p$-equivalence). The main construction in \cite{BZ09} was to realize 
efficiently the fibered join $J^p_{\mathrm{pr}_1}(X)$ up to homotopy by a 
semi-algebraic set. 
This construction however is semi-algebraic in nature, i.e. it uses real
inequalities in an essential way and thus does not generalize in a 
straightforward way to the complex case. Thus, a different construction is 
needed in the complex case.

In the complex case, the role of the fibered join is played by the 
{\em complex join fibered over a projection 
$\mathrm{pr}_1: \C^{m} \times \C^n \rightarrow \C^m$}
defined below (see Definition \ref{def:joinoveramap1}). 
The fibers of the $(p+1)$-fold
complex join fibered over a projection $\mathrm{pr}_1$,
$J^p_{\C}(X)$, of a compact constructible set $X$ are not
quite $p$-connected as in the real case, but are reasonably nice -- namely 
they are homologically equivalent to a projective space of dimension $p$
(see Proposition \ref{prop:equivalence}). 
This
allows us to relate the Poincar\'e polynomial of $X$ with that of its 
image $\mathrm{pr}_1(X)$, 
even though the relation is not as straightforward as in the real
case (see Theorem \ref{the:compactcovering} below).

We remark that Theorem  \ref{the:compactcovering} can be used to express
directly the Betti numbers of the image under projection of a 
projective variety in terms of those another projective variety
obtained directly without having to perform effective quantifier
elimination (which has exponential complexity).
The description of this second variety is {\em much simpler and 
algebraic} in nature compared
to the one used in \cite{BZ09} in the real semi-algebraic case, 
and thus might be of independent interest. 
Theorem  \ref{the:compactcovering}  can also be viewed as an 
improvement over the descent spectral sequence argument used in \cite{GVZ04}
to bound the Betti numbers of  projections (of semi-algebraic sets) in the
complex projective case. 
A similar construction using the projective join is also available in the
real case (using $\Z/2\Z$ coefficients) but we omit its description
in the current paper.

Finally, we believe that 
the compact projective versions of the complex complexity classes introduced
in this paper deserve further investigations on their own 
(see Remark \ref{rem:compact_classes} below), 
since many numerical algorithms for computing solutions of  complex polynomial
systems assume some form of compactness
(see, for instance, \cite{Shub-Smale93, Beltran-Pardo09, buergisser-2009}).

The rest of the paper is organized as follows. In Section \ref{sec:ingredients}
we state and prove the necessary ingredients from algebraic topology
needed to prove the main theorems.
In Section \ref{sec:proof} we prove the main results of the paper.
Finally, in Section \ref{sec:future}, we pose some open problems and
discuss possible extensions to the current work. 
\section{Topological Ingredients}
\label{sec:ingredients}

In this section we state and prove the main topological ingredients
necessary for the proof of the main theorems.

\subsection{Alexander-Lefschetz duality}
We will need the classical Alexander-Lefschetz duality theorem
in order to relate the Betti numbers of 
a 
closed
constructible subset 
$S \subset \PP_\C^{k_1} \times \cdots \times \PP_\C^{k_\ell}$ with those
of its complement
$\PP_\C^{k_1} \times \cdots \times \PP_\C^{k_\ell} \setminus S$.

\begin{theorem}[Alexander-Lefschetz duality]
\label{the:alexanderduality_general}
Let $S \subset \PP_\C^{k_1} \times \cdots \times \PP_\C^{k_\ell}$ be a 
closed
constructible subset.
Then for each odd $i$, $1\leq i \leq 2k+1$ with
$k = k_1 + \cdots + k_\ell$, we have that
\begin{equation}
b_{i-1}(S) - b_{i-2}(S) = b_{2k -i}(\PP_\C^{k_1} \times \cdots 
\times \PP_\C^{k_\ell} - S) - 
b_{2k -i+1 }(\PP_\C^{k_1} \times \cdots \times \PP_\C^{k_\ell} - S) + 
b_{i-1}(\PP_\C^{k_1} \times \cdots \times \PP_\C^{k_\ell}).  
\end{equation}
\end{theorem}
\begin{proof}
Lefshetz duality theorem \cite{Spanier} gives for each $i, 0 \leq i \leq 2k$,
\[
b_i(\PP_\C^{k_1} \times \cdots \times \PP_\C^{k_\ell} - S) = 
b_{2k-i}(\PP_\C^{k_1} \times \cdots \times \PP_\C^{k_\ell}, S).
\]
The theorem now follows from the long exact sequence of homology,
\[
\cdots \rightarrow \HH_i(S) \rightarrow \HH_i(\PP_\C^{k_1} \times \cdots 
\times \PP_\C^{k_\ell}) \rightarrow 
\HH_i(\PP_\C^{k_1} \times \cdots \times \PP_\C^{k_\ell},S) 
\rightarrow \HH_{i-1}(S) \rightarrow \cdots
\] 
after noting that $\HH_i(\PP_\C^{k_1} \times \cdots \times 
\PP_\C^{k_\ell}) = 0,$ for all $i \neq 0,2,4,\ldots,2k$. 
\end{proof}

For technical reasons (see Corollary \ref{cor:alexanderduality_general} below) 
we need to consider
the even and odd parts of the Poincar\'e polynomial of constructible sets.
 
Given $P  = \sum_{i \geq 0} a_i T^i \in \Z[T]$, we write
\[
P \defeq P^{{\mathrm{even}}}(T^2) + T P^{{\mathrm{odd}}}(T^2),
\]
where 
\[
P^{{\mathrm{even}}}(T) = \sum_{i \geq 0} a_{2i} T^i,
\]
and
\[
P^{{\mathrm{odd}}}(T) = \sum_{i \geq 0} a_{2i+1} T^i.
\]

We introduce for any
$S \subset \PP^n_\C$, a 
related polynomial, $Q_S(T)$, which we call
the \bfdef{pseudo-Poincar\'e polynomial} of $S$
defined as follows.

\begin{equation}
\label{def:pseudo-Poincarepolynomial}
Q_S(T)\; \defeq \; \sum_{j \geq 0} (b_{2j}(S) - b_{2j-1}(S))T^j.
\end{equation}

In other words,

\begin{equation}
\label{def:pseudo-Poincarepolynomial2}
Q_S =  P^{{\mathrm{even}}}_S - T P^{{\mathrm{odd}}}_S. 
\end{equation}

We introduce below notation for several operators on polynomials that 
we will use later.

\begin{notation}[Operators on polynomials]
\label{not:poloperators}
For any polynomial $Q = \sum_{i\geq 0} a_i T^i\in \Z[T]$ with 
$\deg(Q) \leq n$, we will denote by:
\begin{enumerate}
\item
$\Rec_n(Q)$ the polynomial $T^n Q(\frac{1}{T})$;
\item
for $0 \leq m \leq n$, 
$\Trunc_m(Q)$ the polynomial $\sum_{0 \leq i \leq m} a_i T^i\in \Z[T]$; and,
\item
$M_P(Q)$ the polynomial $P Q$, for any polynomial $P \in \Z[T]$.
\end{enumerate}
\end{notation}

\begin{remark}
\label{rem:poloperators}
Notice that all the operators introduced above are computable in polynomial
time. 
\end{remark}

Using the notation introduced above
we have the following easy corollary of Theorem 
\ref{the:alexanderduality_general}.

\begin{corollary}
\label{cor:alexanderduality_general}
Let $S \subset \PP_\C^{k_1} \times \cdots \times \PP_\C^{k_\ell}$ 
be 
any  
closed (resp. open)
constructible subset,
and $k = k_1 + \cdots + k_\ell$.
Then,
$$
\displaylines{
Q_S = Q_{\PP_\C^{k_1} \times \cdots \times \PP_\C^{k_\ell}}
- \Rec_k(Q_{\PP_\C^{k_1} \times \cdots \times \PP_\C^{k_\ell} - S}).
}
$$
\end{corollary}

\subsection{The complex join of subsets of complex projective spaces}
We first give a purely geometric definition of the complex join
of two sets followed by one using co-ordinates.
The geometric definition is useful in understanding the topological
properties of the join proved later. The definition involving
co-ordinates and formulas is necessary for the 
complexity theoretic arguments.

Let $V,W$ be finite dimensional vector spaces over $\C$ and 
let $X \subset \PP(V)$ and $Y \subset \PP(W)$ be two arbitrary 
(not necessarily constructible) 
subsets. Note that 
$\PP(V) \cong \PP(V\oplus \mathbf{0}) \subset \PP(V\oplus W)$ 
and $\PP(W) \cong \PP(\mathbf{0}\oplus W) \subset \PP(V\oplus W)$ are two
disjoint subspaces of $\PP(V \oplus W)$ and thus $X$ and $Y$ are 
embedded as disjoint subsets of $\PP(V\oplus W)$.
With the above notation

\begin{definition}[Geometric definition of complex join]
\label{def:complexjoin_geometric}
The \bfdef{complex join}, \bfdef{$J_\C(X,Y)$}, is 
defined to be the union of projective
lines in $\PP(V\oplus W)$ which meets both $X$ and $Y$ if $X$ and
$Y$ are both non-empty. We let  $J_\C(X,Y) = Y$ if $X$ is empty,
and  $J_\C(X,Y) = X$ if $Y$ is empty.
\end{definition}

We now give a definition of the complex join which 
involve co-ordinates which we are going to use in this paper. 

Let $X \subset \PP^k_\C$ and $Y \subset \PP^\ell_\C$ be two constructible
sets defined by homogeneous formulas $\Phi(X_0,\ldots,X_k)$ and 
$\Psi(Y_0,\ldots,Y_\ell)$
respectively, where $(X_0:\cdots:X_k)$ (respectively $(Y_0:\cdots:Y_\ell)$)
are homogeneous co-ordinates in $\PP^k_\C$ (respectively $\PP^\ell_\C$).

\begin{definition}[Complex join in terms of co-ordinates]
\label{def:complexjoin}
The complex join, $J_\C(X,Y)$, of $X$ and $Y$ is 
the constructible subset of $\PP^{k+\ell+1}_\C$ defined by the formula
\bfdef{
\[
J_\C(\Phi,\Psi) \defeq 
\phi(Z_0,\cdots,Z_k) \wedge \psi(Z_{k+1},\cdots,Z_{k+\ell+1}),
\]
}
where $(Z_0:\cdots:Z_{k+\ell+1})$ are  homogeneous coordinates
in $\PP^{k+\ell+1}_\C$.
\end{definition}

\begin{remark}
\label{rem:empty}
Firstly, notice that 
the realization, $\RR(J_\C(\Phi,\Psi))$,
does not depend on the formulas
$\phi$ and $\psi$ used to define $X$ and $Y$ respectively.
Also, notice that 
if $X$ and $Y$ are {\em both} empty then so is $J_\C(X,Y)$. 
Indeed, if $X = \emptyset$ (respectively, $Y = \emptyset$) then $J_\C(X,Y)$ is
isomorphic to $Y$ (respectively, $X$). 
To see this notice that
by definition (cf. Definition \ref{def:multi-homogeneous}) 
the homogeneous formula
$\Phi(\X)$ (resp. $\Psi(\Y)$) is true whenever  $\X$ (resp. $\Y$) is the 
$\mathbf{0}$-vector.
Now consider the following two constructible subsets of $\PP_\C^{k+\ell+1}$.
$$
\displaylines{
\tilde{X} = \{(\x:0:\cdots:0)\; \mid \;\x \in X \}, \cr
\tilde{Y} = \{(0:\cdots:0:\y)\; \mid \;\y \in Y \}.
}
$$

We have that $\tilde{X}$ (resp. $\tilde{Y}$) is isomorphic to $X$ 
(resp. $Y$). Moreover $\tilde{X}$ and $\tilde{Y}$ are
contained in $J_\C(X,Y)$, 
since as remarked earlier
$\Psi(\Y)$ (resp. $\Phi(\X)$) is true whenever $\Y$ (resp. $\X$) is the 
$0$-vector.
Moreover, $\tilde{X}$ (resp. $\tilde{Y}$) is equal to $J_\C(X,Y)$ 
in case $Y$ (resp. $X$) is empty.

Also, clearly $\tilde{X}$ and $\tilde{Y}$ are disjoint, and
if $X$ and $Y$ are both non-empty then, 
$J_\C(X,Y)$ is obtained by 
taking the union of projective lines in $\PP^{k+\ell+1}_\C$
meeting both $\tilde{X}$ and $\tilde{Y}$.
\end{remark}

\begin{example}
\label{eg:joinoftwoprojectives}
It is easy to check from the above definition that
the join, $J_\C(\PP^k_\C,\PP^\ell_\C)$, of two projective spaces
is again a projective space, namely $\PP^{k+\ell+1}_\C$.
\end{example}

\begin{remark}
The projective join as defined above is a classical object in algebraic
geometry. Amongst many other applications, 
the complex suspension of a projective
variety $X$ (i.e. the complex join 
$J_\C(X,\PP^0_\C)$)
plays an important role in defining
Lawson homology of projective varieties \cite{Lawson}.
Within the area of computational complexity theory,
the projective join of a variety with a point was used in
\cite{Scheiblechner07} for proving hardness of the problem of
computing Betti numbers of complex varieties. 
\end{remark}

\begin{definition}
\label{def:p-foldcomplexjoin}
For 
$p \geq 0$, 
we denote by \bfdef{$J_\C^p(X)$} the $(p+1)$-fold iterated
complex join of $X$ with itself.

More precisely,
$$
\displaylines{
 J_\C^0( X) := X, \cr
 J_\C^{p+1}(X) := J_\C(J_\C^{p}(X),X), \mbox{ for } p \geq 1.
}
$$

If $X \subset \PP^k_\C$ is defined by a first-order homogeneous formula 
$\Phi(X_0,\ldots,X_k)$, 
then 
$J^p_\C(X) \subset \PP^{(p+1)(k+1) -1 }_\C$ is defined by the homogeneous formula
$$
\displaylines{
J^p_\C(\Phi)(X^0_0,\ldots,X^0_k,\ldots,X^p_0,\ldots,X^p_{k}) \defeq 
\bigwedge_{i=0}^{p}\phi(X^i_0,\ldots,X^i_k).
}
$$
where $(X^0_0:\cdots:X^p_{k})$ are homogeneous co-ordinates in
$\PP^{(p+1)(k+1) -1 }_\C$.
\end{definition}

Note that by  Remark \ref{rem:empty}, if $X$ is empty then $J^p_\C(X)$ is empty
for every $p \geq  0$.

\subsection{Properties of the topological join}

We also need to introduce the \bfdef{topological join}
of two spaces. The following is 
classical.

\begin{definition}
\label{def:twofoldjoin}
The \bfdef{join}, \bfdef{$X*Y$},  
of two topological spaces $X$ and $Y$ is defined by
\begin{equation}
\label{eqn:definitionoftwofoldjoin}
X*Y \defeq X\times Y 
\times \Delta^1/\sim,  
\end{equation}
where
$\Delta^1 = \{(t_0,t_1) \;\mid\; t_0,t_1 \geq 0, t_0+t_1=1\}$ 
denotes the standard geometric realization of the $1$-dimensional simplex, and 
\[
(x,y,t_0,t_1) \sim (x',y',t_0,t_1)
\] 
if and only if
$t_0 = 1,x = x'$ or  $t_1=1, y= y'$.
\end{definition}

Intuitively, $X*Y$ is obtained by joining each point of $X$ with
each point of $Y$ by an interval.

We will need the well-known 
fact that the iterated join of a topological space is 
highly connected. In order to make this statement precise we first define

\begin{definition}[$p$-equivalence]
\label{def:p-equivalence}
A map $f: A \rightarrow B$ between two topological spaces is called a 
\bfdef{$p$-equivalence} if the induced homomorphism 
\[
f_*: \HH_i(A) \rightarrow \HH_i(B)
\]
is an isomorphism for all $0 \leq i < p$, and an epimorphism for $i=p$,
and we say that $A$ is \bfdef{$p$-equivalent} to $B$.
\end{definition}
 
The following is well known.
(see, for instance, 
\cite[Proposition 4.4.3]{Matousek_book2}).

\begin{theorem}
\label{the:pjoinconnectivity}
Let $X$ be a 
non-empty
compact semi-algebraic set. Then, the $(p+1)$-fold join
$ \underbrace{{X} * \cdots * {X}}_{(p+1) \mbox{ times}}$ is 
$p$-equivalent to a point.
\end{theorem}

We will need a particular property of projection maps that 
we are going to consider later in the paper.

\begin{notation}
For any constructible set $A$, we denote by 
$\K(A)$ the
collection of all compact (in the Euclidean topology) subsets of $A$.
\end{notation}

\begin{definition}
\label{def:compactcovering}
Let $f: A \to B$ be a map between two constructible sets  $A$ and $B$. 
We say that $f$ \bfdef{compact covering} if for any $L\in \K(f(A))$, there
exists $K\in \K(A)$ such that $f(K)=L.$
\end{definition}

\subsection{Topological properties of the complex join}

\begin{proposition}
\label{prop:equivalence}
Let $X \subset \PP^k_\C$ be a 
non-empty  
semi-algebraic
subset and $p > 0$. Let
\[
\iota : J_\C^{p} (X) \hookrightarrow \PP^{(p+1)(k+1)-1}_\C
\] 
denote the inclusion map.
Then the induced homomorphisms
$$
\displaylines{
\iota_*: \HH_j(J_\C^p (X)) \rightarrow \HH_j(\PP^{(p+1)(k+1)-1}_\C) \cr
\iota^*: \HH^j(\PP^{(p+1)(k+1)-1}_\C) \rightarrow \HH^j(J_\C^p (X))
}
$$
are isomorphisms for $0 \leq j < p$.
\end{proposition}

Before proving Proposition \ref{prop:equivalence}
we first fix some notation.
\begin{notation}[Hopf fibration]
For any $k \geq 0$, we will denote by 
$\pi:\C^{k+1}\setminus \{0\}\rightarrow \PP^k_\C$ the tautological line 
bundle over $\PP^k_\C$, and by 
\[
\tilde{\pi}: \Sphere^{2k+1}\rightarrow \PP^k_\C,
\]
the \bfdef{Hopf fibration}, 
namely the restriction of $\pi$ to the unit sphere in
$\C^{k+1}$ defined by the equation $|z_1|^2 + \cdots + |z_{k+1}|^2 = 1$.
Finally for any subset $S \subset \PP^k_\C$, we will denote by 
$\widetilde{S}$ the subset $\tilde{\pi}^{-1}(S) \subset \Sphere^{2k+1}$. 
Restricting the map $\tilde{\pi}$  to $\widetilde{S}$ 
we obtain the restriction of the Hopf fibration to the base $S$ 
i.e. we have the following commutative diagram.

\[
\xymatrix{
\widetilde{S}  \ar@{^{(}->}[r]^{\iota} \ar[d]^{\tilde{\pi}}& \Sphere^{2k+1} \ar[d]^{\tilde{\pi}} \\
S \ar@{^{(}->}[r]^{\iota} & \PP^{k}_\C 
}
\]
\end{notation}

We need the following lemma.
\begin{lemma}
\label{lem:Hopf}
Let $X \subset \PP^k_\C, Y \subset \PP^\ell_\C$ be  
semi-algebraic
subsets. Then 
$\widetilde{J_\C (X,Y)} \subset \Sphere^{2(k+\ell) + 3}$ is homeomorphic to 
the  (topological) join  $ \widetilde{X} * \widetilde{Y}$.
\end{lemma}
\begin{proof}
Consider $\x \in X$ and $\y \in Y$ and the projective line 
$L \subset J_\C(X,Y)$
joining $\x$ and $\y$. It is easy to see that the 
preimage $\tilde{L} = \tilde{\pi}^{-1}(L) \cong \Sphere^{3}$ is a 
topological join of
$\tilde{\pi}^{-1}(\x)$ and $\tilde{\pi}^{-1}(\y)$ 
(each homeomorphic to $\Sphere^1$). Now
since $\tilde{X}$ (resp. $\tilde{Y}$) is fibered by the various 
$\tilde{\pi}^{-1}(\x)$ 
(resp. $\tilde{\pi}^{-1}(\y)$), it follows that $\widetilde{J_\C (X,Y)}$ is homeomorphic
to $ \widetilde{X} * \widetilde{Y}$.
\end{proof}

\begin{proof}[Proof of Proposition \ref{prop:equivalence}]
We first treat the cases $p=1,2$. 
\begin{enumerate}
\item[$p=1$:] It is an easy exercise to show that the join,
$J_\C^1(X) = J_\C(X,X)$ is non-empty and connected whenever $X$ is
non-empty.  This proves the proposition in this case.
\item[$p=2$:] It is easy to see that 
$J_\C^2(X) = J_\C(J^1_\C(X),X)$ is non-empty and connected, whenever $X$ 
is non-empty. It is only a slightly more difficult exercise to prove that
$\HH_1(J_\C^2(X))$ (and in fact, $\HH_1(J_\C^p(X))$ for all $p > 1$)
vanishes. This follows 
from the statement that $\HH_1(J_\C(Y,Z)) = 0$ whenever $Y$ is connected,
since $J_\C^p(X) = J_\C(J_\C^{p-1}(X),X)$ and we have that 
$J_\C^{p-1}(X)$ is connected for $p > 1$.  
Proving that $\HH_1(J_\C(Y,Z)) = 0$ whenever $Y$ is connected,
after an application of Mayer-Vietoris exact sequence,
reduces to proving that $\HH_1(J_\C(Y,Z)) = 0$ whenever both
$Y$ and $Z$ are connected. This can be checked by a direct calculation
using the fact that the topological join $Y * Z$ is simply connected
whenever $Y,Z$ are both connected. Note that, this also proves that
$J^p_\C(X)$ is simply connected for all $p > 1$.
\end{enumerate}

Now let $p \geq  2$.
It follows from repeated applications of Lemma  \ref{lem:Hopf} that
$\widetilde{J_\C^p(X)}$ is homeomorphic to 
$$\displaylines{
\underbrace{ \widetilde{X} * \cdots * \widetilde{X}}_{(p+1) \mbox{ times}}.
}
$$

We also have the commutative square
\[
\xymatrix{
\widetilde{J_\C^p (X)}  \ar@{^{(}->}[r]^-i \ar[d]^{\tilde{\pi}}& \Sphere^{2(p+1)(k+1)-1} \ar[d]^{\tilde{\pi}} \\
J_\C^p (X) \ar@{^{(}->}[r]^-i & \PP^{(p+1)(k+1)-1}_\C 
}
\]
and a corresponding square 
\[
\xymatrix{
\HH_*(\widetilde{J_\C^p (X)})  \ar[r]^-{\iota_*} \ar[d]^{\tilde{\pi}_*}& 
\HH_*(\Sphere^{2(p+1)(k+1)-1}) \ar[d]^{\tilde{\pi}_*} \\
\HH_*(J_\C^p (X)) \ar[r]^-{\iota_*} & \HH_*(\PP^{(p+1)(k+1)-1}_\C) 
}
\]
of induced homomorphisms in the homology groups.

It follows from Theorem \ref{the:pjoinconnectivity} that 
if $X \neq \emptyset$, then
$$
\displaylines{
\HH_0(\widetilde{J_\C^p(X)})  \cong \Q,\cr
\HH_i(\widetilde{J_\C^p(X)}) \cong 0, \;\;
0 < i < p. 
}
$$

Since, for $p > 1$, $J_\C^p(X)$ is simply connected (see above) 
$\widetilde{J_\C^p(X)}$ is a simple $\Sphere^1$-bundle 
(i.e. a $\Sphere^1$-bundle with a simply connected base)
over $J_\C^p(X)$.

It now follows by a standard argument (which we expand below) 
involving the spectral sequence of the
bundle $\tilde{\pi}:\widetilde{J_\C^p (X)} \rightarrow J_\C^p(X)$,
that for $0 \leq i < p$, 

\begin{eqnarray}
\label{eqn:sigma}
\HH_{i}(J_\C^p(X)) &\cong & \Q, \mbox{ for } i \mbox{ even},\\ \nonumber
\HH_{i}(J_\C^p(X)) &\cong & 0 \mbox{ for } i \mbox{ odd}.\\ \nonumber
\end{eqnarray}

(The above claim also follows from the Gysin sequence of the
$\Sphere^1$-bundle $\tilde{\pi}:\widetilde{J_\C^p(X)} \rightarrow J_\C^p(X)$
but we give an independent proof below).

Consider the $E^2$-term of the (homological) 
spectral sequence of the bundle 
\[
\tilde{\pi}:\widetilde{J_\C^p(X)} \rightarrow J_\C^p(X).
\]
For $i,j \geq 0$, we have that
\[
E^2_{i,j} = \HH_i(J_\C^p(X)) \otimes \HH_j(\Sphere^1).
\]

From this we deduce that
\[
E^2_{i,0} = E^2_{i,1} = \HH_i(J_\C^p(X)).
\]

Also, from the fact that
\[\HH_0(\widetilde{J_\C^p(X)}) = \Q,
\]
we get  that 
\[
E^2_{0,0} = \Q,
\] 
and hence, 
\[
E^2_{0,1} = \Q
\] 
as well. 
Moreover, we have that 
\[
E^3_{i,j} = E^4_{i,j} = \cdots = E^\infty_{i,j}
\] 
for all $i \geq 0$ and $j = 0,1$.
Now from the fact that the spectral sequence $E^r$
converges to the homology of $\widetilde{J_\C^p(X)}$ we deduce
that
\begin{eqnarray*}
E^3_{i,j} &=& 0 \mbox{ for } 0 \leq i \leq p-1 \mbox{ and all } j,\\
E^3_{0,0} &=& \Q.
\end{eqnarray*}

This implies that the differential 
$$
d_2: E^2_{i,0} \rightarrow E^2_{i-2,1}
$$
is an isomorphism for $1 \leq i \leq p-1$. 
Together with the fact that
\[
E^2_{i,0} = E^2_{i,1} = \HH_i(J_\C^p(X)),
\] 
this immediately implies (\ref{eqn:sigma}).
The claim that $\iota_*$ is an isomorphism follows directly from the above.
The dual statement about $\iota^*$ follows immediately from the universal
coefficient theorem for cohomology (see e.g. \cite{Spanier}).
\end{proof}

In our application we will need the following 
(rather technical) generalization of Proposition \ref{prop:equivalence}.
Let $p, \alpha_0,\ldots, \alpha_\omega \geq 0$,
$N  = \prod_{0 \leq j \leq \omega} (\alpha_j + 1)$.
Let $I$ denote the set of tuples
$(i_0,\ldots,i_\omega)$ with $0 \leq i_j \leq \alpha_j, 
0 \leq j \leq \omega$,
and for each tuple $(i_0,\ldots,i_\omega) \in I$,
let $\pi_{(i_0,\ldots,i_\omega)}$ denote the projection
\[
\underset{(j_0,\ldots,j_\omega) \in I}{\bigtimes} \PP_\C^k  \longrightarrow \PP_\C^k
\]
defined by 
\[
(\x_{(j_0,\ldots,j_{\omega})})_{(j_0,\ldots,j_\omega) \in I} \mapsto
\x_{(i_0,\ldots,i_\omega)},
\]
and 
for any subset $X \subset \PP^k_\C$
we denote
\[
X^{(i_0,\ldots,i_\omega)} = \pi_{(i_0,\ldots,i_\omega)}^{-1}(X).
\]

\begin{proposition}
\label{prop:equivalence_general}
Let $X \subset \PP_\C^k$ be a semi-algebraic subset.
Also, let for each 
$i, 0 \leq i \leq \omega$, ${\Lambda}^i \in \{\bigcap,\bigcup\}$,  
and let
$
\displaystyle{
S \subset \underset{(j_0,\ldots,j_\omega) \in I}{\bigtimes} \PP_{\C}^k
}
$ 
denote the  semi-algebraic subset
$$
\displaylines{
S \defeq
\underset{0 \leq i_0 \leq \alpha_0}{{\Lambda}^0} \cdots 
\underset{0 \leq i_\omega \leq \alpha_\omega}{{\Lambda}^\omega}
(J^p_\C(X))^{(i_0,\ldots,i_\omega)},
}
$$
with 
\[
\iota: S \hookrightarrow \underset{(j_0,\ldots,j_\omega) \in I}{\bigtimes}
\PP^{(p+1)(k+1)-1}_\C
\]
denoting  the inclusion map. 
Then, the induced homomorphisms
$$
\displaylines{
\iota_{*}: \HH_j(S) \rightarrow \HH_j(\underset{(j_0,\ldots,j_\omega) \in I}{\bigtimes}
\PP^{(p+1)(k+1)-1}_\C) \cr
\iota^{*}: \HH^j(\underset{(j_0,\ldots,j_\omega) \in I}{\bigtimes}
\PP^{(p+1)(k+1)-1}_\C) \rightarrow \HH^j(S)
}
$$
are isomorphisms for $0 \leq j < p$. 
\end{proposition}

\begin{proof}
Notice that, if $\omega=0$ and 
$\displaystyle{{\Lambda}^0 = \bigcap}$, 
then 
$$
\displaylines{
\bigcap_{0 \leq i_0 \leq \alpha_0} J^p_\C(X)^{(i_0)} = \underset{(j_0,\ldots,j_\omega) \in I}{\bigtimes}
J^p_\C(X), 
}
$$
and the claim follows in this case from 
Proposition \ref{prop:equivalence} and the
Kunneth formula.

If $\omega=0$ and $\displaystyle{{\Lambda}^0 = \bigcup}$, 
the claim follows from the 
previous case and a standard argument using the Mayer-Vietoris double
complex.

The general case is easily proved using induction on $\omega$. 
\end{proof}

\subsection{Complex join fibered over a projection and its properties}
In our application we need the complex join fibered over 
certain  projections.
We first give a geometric definition followed by one involving co-ordinates.

Let $V,W$ be finite dimensional $\C$-vector spaces and 
$A \subset \PP(V) \times \PP(W)$ a subset.
Let $\mathrm{pr}_1: \PP(V) \times \PP(W) \rightarrow \PP(V)$ 
denote the projection on the first component. Then, for $p \geq 0$,
the \bfdef{$p$-fold complex join of $A$ fibered over the 
projection $\mathrm{pr}_1$} is defined by

\begin{definition}[Geometric definition of complex join fibered over
a projection]
\label{def:joinoveramap1_geometric}
\begin{eqnarray*}
\label{eqn:definitionofjoin1_geometric}
J^p_{\C,\mathrm{pr}_1}(A) &=& \{(\x,\y) \;\mid \; 
\x \in \PP^k_\C, \y \in J^p_\C(A_\x)\},
\end{eqnarray*}
were $A_\x = \mathrm{pr}_1^{-1}(\x) \cap A$.
\end{definition}

We now give a definition in terms of co-ordinates.
\begin{definition}[Complex join fibered over a projection in terms of 
co-ordinates]
\label{def:joinoveramap1}
Let $A \subset \PP^k_\C \times \PP^{\ell}_\C$ be a constructible  set
defined by a first-order multi-homogeneous formula,
\[
\Phi(X_0,\ldots,X_k;Y_0,\ldots,Y_\ell)
\] 
and let
$\mathrm{pr}_1: 
\PP^k_\C \times \PP^{\ell}_\C \rightarrow \PP^k_\C$ be the projection
map to the first component.
For $p \geq 0$, the $p$-fold complex join of $A$ fibered over the 
map $\mathrm{pr}_1$, 
$J^p_{\C,\mathrm{pr}_1}(A) \subset \PP^{k}_\C \times \PP^{(\ell+1)(p+1)-1}_\C $,  
is defined by  the formula
\bfdef{
\begin{equation}
\label{eqn:definitionofjoin1}
J^p_{\C,\mathrm{pr}_1}(\Phi)(X_0,\ldots,X_k;Y^0_0,\ldots,Y^0_\ell,\ldots,Y^p_0,
\ldots,Y^p_{\ell}) 
\defeq
\bigwedge_{i=0}^p \phi(X_0,\ldots,X_k;Y^i_0,\ldots,Y^i_\ell). 
\end{equation}
}
\end{definition}

\begin{remark}
\label{rem:induced}
The projection map 
\[
\mathrm{pr}_1: \PP^{k}_\C \times \PP^{(\ell+1)(p+1)-1}_\C \rightarrow
\PP^{k}_\C
\]
clearly restricts to a surjection 
\[
\mathrm{pr}_1^{(p)}: J^p_{\C,\mathrm{pr}_1}(A) \rightarrow \mathrm{pr}_1(A)
\] 
sending 
$(x_0:\cdots:x_k;y^0_0:\cdots:y^p_\ell) \in J^p_{\C,\mathrm{pr}_1}(A)$ to
$(x_0:\cdots:x_k) \in \mathrm{pr}_1(A)$. 
\end{remark}

Now, let $A \subset \PP^k_\C \times \PP^\ell_\C$ be a 
semi-algebraic
subset 
$\mathrm{pr}_1:  \PP^k_\C \times \PP^\ell_\C \rightarrow \PP^k_\C$ be the
projection on the first component.

Suppose that $\mathrm{pr}_1$ restricted to $A$ is a 
compact covering.
The following theorem relates the Poincar\'e polynomial of 
$J^p_{\C,\mathrm{pr}_1}(A)$ to that of the image  $\mathrm{pr}_1(A)$.

\begin{theorem}
\label{the:compactcovering}
For every $p \geq 0$, we have that
\begin{eqnarray}
\label{eqn:complexjoin}
P_{\mathrm{pr}_1(A)}  &=&  (1 - T^2) P_{J^p_{\C,\mathrm{pr}_1}(A)}  \mod T^{p}.
\end{eqnarray}
\end{theorem}

\begin{remark}
\label{rem:compact_covering}
Note that the compact covering property is crucial for Theorem
\ref{the:compactcovering}.
to hold. In our applications,  $\mathrm{pr}_1$
is going to be either an open or a closed map, and will thus automatically
have the compact covering property.
\end{remark}

\begin{proof}
We first assume that $A$ is semi-algebraic and compact, and let $B$ denote 
$\mathrm{pr}_1(A) \times \PP_\C^{(p+1)(\ell+1)-1}$.
We have the following commutative square.

\[
\xymatrix{
J^p_{\C,\mathrm{pr}_1}(A)  \ar@{^{(}->}[r]^-i \ar[d]^{\mathrm{pr}_1^{(p)}}& 
B \ar[d]^{\mathrm{pr}_1} \\
\mathrm{pr}_1(A) \ar[r]^{{\mathrm{Id}}} & \mathrm{pr}_1(A) 
}
\]

The diagram above induces a morphism, 
$\phi_r^{i,j}: E_r^{i,j} \rightarrow {'E}_r^{i,j}$
between the Leray-Serre spectral sequences
of the two vertical maps in the above diagram. Here, $E_r$ (resp. ${'E}_r$)
denotes the Leray-Serre spectral sequence of the map 
$\mathrm{pr}_1: B \rightarrow 
\mathrm{pr}_1(A)$
(resp. 
$\mathrm{pr}_1^{(p)}: J^p_{\C,\mathrm{pr}_1}(A) \rightarrow \mathrm{pr}_1(A)$ 
).
The spectral sequence, ${E}_r$, 
degenerates at the ${E}_2$-term
where 
\[
{E}_2^{i,j} = \HH^i(\mathrm{pr}_1(A), R^j\mathrm{pr}_{1*}\Q_B),
\]
and $\Q_B$ 
denotes the constant sheaf with stalk $\Q$ on 
$B$,  and
$R^*\mathrm{pr}_{1*}$ denotes  the higher direct image functor.
(The above formulation of Leray-Serre spectral sequence of a map is 
standard; we refer the reader to \cite[Th\'eor\`eme  4.17.1]{Godement} for a 
purely sheaf theoretic statement without reference to higher derived
images.)

Similarly we have
\[
{'E}_2^{i,j} = \HH^i(\mathrm{pr}_1(A), R^j\mathrm{pr}_{1*}^{(p)} \Q_{J^p_{\C,\mathrm{pr}_1}(A)}).
\]

We also have that for each $\x \in \mathrm{pr}_1(A)$
$$\displaylines{
(R^j\mathrm{pr}_{1*}\Q_B)_\x\cong 
\HH^j(\PP^{(p+1)(\ell+1)-1}_\C)
\cong
\HH^j((\mathrm{pr}_1^{(p)})^{-1}(\x))
\cong
(R^j\mathrm{pr}_{1*}^{(p)} \Q_{J^p_{\C,\mathrm{pr}_1}(A)})_\x,
}
$$
where the first and the last isomorphisms are consequences of the
proper base change theorem 
(see \cite[Remarque 4.17.1]{Godement}) 
noting that $\mathrm{pr}_1, \mathrm{pr}_1^{(p)}$ are both
proper maps,
and the middle one is  a consequence
of Proposition \ref{prop:equivalence}.

It follows that the sheaves 
$R^j\mathrm{pr}_1\Q_B$ and
$R^j\mathrm{pr}_1^{(p)} \Q_{J^p_{\C,\mathrm{pr}_1}(A)}$
are isomorphic by the sheaf map induced by the inclusions
\[
(\mathrm{pr}_1^{(p)})^{-1}(\x)
\hookrightarrow
\{\x\} \times \PP^{(p+1)(\ell+1)-1}_\C, \x \in \mathrm{pr}_1(A)
\]
and hence,
\[
\phi_2^{i,j}:E_2^{i,j} \rightarrow {'E}_2^{i,j}
\] 
are isomorphisms for $i+j < p$. 

It now follows from a general result about spectral sequences 
(see \cite[page. 66]{Mcleary}) that
$E_\infty^{i,j} \cong  {'E}_\infty^{i,j}$ for $0 \leq i+j < p$.
This implies that
$\HH^q(J^p_{\C,\mathrm{pr}_1}(A)) \cong \HH^q(\mathrm{pr}_1(A) \times \PP_\C^{(p+1)(\ell+1)-1})$
for $0 \leq q < p$, and thus
\begin{equation}
\label{eqn:1}
P_{J^p_{\C,\mathrm{pr}_1}(A)} = P_{\mathrm{pr}_1(A) \times \PP_\C^{(p+1)(\ell+1)-1}} \mod T^{p}.
\end{equation}

We also have that
\begin{eqnarray}
\label{eqn:2}
P_{\mathrm{pr}_1(A) \times \PP_\C^{(p+1)(\ell+1)-1}} &=& P_{\mathrm{pr}_1(A)} \times P_{\PP_\C^{(p+1)(\ell+1)-1}} \\ \nonumber
                               &=& P_{\mathrm{pr}_1(A)} \times (1 + T^2 +\cdots+ T^{2((p+1)(\ell+1)-1)}) \\ \nonumber
                               &=& P_{\mathrm{pr}_1(A)} \times (1 - T^2)^{-1} \mod T^{p}.
\end{eqnarray}

Equation 
(\ref{eqn:complexjoin})
now follows from Equations (\ref{eqn:1}) and (\ref{eqn:2}).

The general case follows by taking direct limit 
over all compact 
subsets of $A$. More precisely, for $K_1 \subset K_2$
compact subsets of $A$, we have for
$0 \leq q < p$ the following commutative square after switching
to homology (cf. Remark \ref{rem:coefficients}).
(Note that following Definition 
\ref{def:joinoveramap1_geometric}
the complex join fibered over a projection is defined for
arbitrary not necessarily constructible subsets of $A$.)

\[
\xymatrix{
\HH_q(J_{\C,\mathrm{pr}_1}^p(K_1)) \ar[r]^{\iota_*} \ar[d]^{\cong}& 
\HH_q(J_{\C,\mathrm{pr}_1}^p(K_2))\ar[d]^{\cong} \\
\HH_q(\mathrm{pr}_1(K_1)\times \PP_\C^{(p+1)(\ell+1)-1})\ar[r]^{\iota_*} & 
\HH_q(\mathrm{pr}_1(K_2)\times \PP_\C^{(p+1)(\ell+1)-1}) \\
}
\]
where the vertical maps are isomorphisms  by the previous case. If  we take the
direct limit as $K$ ranges in $\K(A)$, we obtain the following:
\[
\xymatrix{
\dlim \HH_q(J_{\C,\mathrm{pr}_1}^p(K)) \ar[r]^{\cong} \ar[d]^{\cong}& 
\HH_q(J_{\C,\mathrm{pr}_1}^p(A))\ar[d] \\
\dlim \HH_q(\mathrm{pr}_1(K)\times \PP_\C^{(p+1)(\ell+1)-1})\ar[r]^{\cong} & 
\HH_q(\mathrm{pr}_1(A)\times \PP_\C^{(p+1)(\ell+1)-1}) \\
}
\]
The isomorphism on the top level comes from the fact that homology and
direct limit commute~\cite{Spanier}.
For the bottom isomorphism, we need
the additional fact that since we assume that $\mathrm{pr}_1$ 
is 
a compact covering
we have 
$$ \dlim \{\HH_q(\mathrm{pr}_1(K)\times \PP_\C^{(p+1)(\ell+1)-1}) 
\mid K\in\K(A)\}
=\dlim \{\HH_q(L\times \PP_\C^{(p+1)(\ell+1)-1} ) \mid L\in \K(\mathrm{pr}_1(A))\}. 
$$ 
This proves that the right vertical arrow is also an isomorphism.
\end{proof}

Using the same notation as in Theorem \ref{the:compactcovering} 
and Eqn. \eqref{def:pseudo-Poincarepolynomial}
we have 
the following easy corollary of Theorem \ref{the:compactcovering}.
\begin{corollary}
\label{cor:compactcovering}
Let $p = 2m +1$ with $m \geq 0$. Then
\begin{eqnarray}
Q_{\mathrm{pr}_1(A)}  &=&  (1 - T)\;Q_{J^p_{\C,\mathrm{pr}_1}(A)}  \mod T^{m+1}.
\end{eqnarray}
\end{corollary}

\begin{proof}
The corollary follows directly from Theorem \ref{the:compactcovering} and
the fact that for any polynomial $P \in \Z[T]$ we have 
\begin{eqnarray*}
((1 - T^2) P)^{\mathrm{even}} &=& (1-T) (P)^{\mathrm{even}}, \\
((1 - T^2) P)^{\mathrm{odd}} &=& (1-T) (P)^{\mathrm{odd}}.
\end{eqnarray*}
\end{proof}

As before we need a slightly more general version of 
Theorem \ref{the:compactcovering}  as well as Corollary
\ref{cor:compactcovering}. 

Let $\alpha_0,\ldots, \alpha_\sigma \geq 0$,
and 
$N = \prod_{0 \leq j \leq \omega} (\alpha_j+1)$. 
Let $\phi$ be a 
homogeneous formula defining a constructible subset of $\PP^{k_0} \times \cdots
\times \PP^{k_{\sigma}}_\C\times \PP^\ell_\C$.
Also, let for each 
$i, 0 \leq i \leq \sigma$, ${\Lambda}^i \in \{\bigvee,\bigwedge\}$, and let
$\Phi$ denote the multi-homogeneous formula defined by 
$$
\displaylines{
\Phi \defeq
\underset{0 \leq i_0 \leq \alpha_0}{{\Lambda}^0} \cdots 
\underset{0 \leq i_\sigma \leq \alpha_\sigma}{{\Lambda}^\sigma}
\phi(\X^0;\cdots;\X^{\sigma}; \Y_{i_0,\ldots,i_\sigma}).
}
$$
Let 
\[
A = \RR(\Phi) \subset \PP^{k_0} \times \cdots
\times \PP^{k_{\sigma}}_\C \times \underbrace{\PP^{\ell}_\C \times \cdots\times 
\PP^{\ell}_\C}_{N \mbox{ times }},
\]
and let
$\mathrm{pr}_{[0,\sigma]}: 
\PP^{k_0}_\C \times \cdots
\times \PP^{k_{\sigma}}_\C \times \underbrace{\PP^{\ell}_\C \times \cdots\times 
\PP^{\ell}_\C}_{N \mbox{ times }}
\rightarrow \PP^{k_0}_\C \times \cdots
\times \PP^{k_{\sigma}}_\C$ 
be the projection onto the first $\sigma+1$ components, 
and suppose that $\mathrm{pr}_{[0,\sigma]}$ restricted to $A$ is a 
compact covering.

For $p \geq 0$, let
$$
\displaylines{
J^p_{\C,\mathrm{pr}_{[0,\sigma]}}(A) \subset \PP^{k_0} \times \cdots
\times \PP^{k_{\sigma}}_\C \times 
\underbrace{\PP^{(\ell+1)(p+1)-1}_\C \times \cdots\times 
\PP^{(\ell+1)(p+1)-1}_\C}_{N \mbox{ times }}
}
$$  
be defined by  the formula

\begin{equation}
\label{eqn:definitionofjoin2}
J^p_{\C,\mathrm{pr}_{[0,\sigma]}}(\Phi) \defeq  
\underset{0 \leq i_0 \leq \alpha_0}{{\Lambda}^0} \cdots 
\underset{0 \leq i_\sigma \leq \alpha_\sigma}{{\Lambda}^\omega}
J^p_{\C,\mathrm{pr}_{[0,\sigma]}} \phi(\X^0;\cdots;\X^{\sigma};\Y_{i_0,\ldots,i_\sigma}).
\end{equation}

\begin{theorem}
\label{the:compactcovering_general}
For every $p \geq 0$, we have that
\begin{eqnarray}
\label{eqn:complexjoin2}
P_{\mathrm{pr}_{[0,\sigma]}(A)}  &=&  (1 - T^2)^N P_{J^p_{\C,\mathrm{pr}_{[0,\sigma]}}(A)}  \mod T^{p}.
\end{eqnarray}
\end{theorem}

\begin{proof}
The proof is identical to that of Theorem 
\ref{the:compactcovering} above using 
Proposition \ref{prop:equivalence_general}
instead of Proposition \ref{prop:equivalence} and noticing that by
the Kunneth formula for homology, 
the Poincar\'e polynomial of 
\[
\underbrace{\PP^{(\ell+1)(p+1)-1}_\C \times \cdots\times 
\PP^{(\ell+1)(p+1)-1}_\C}_{N \mbox{ times }}
\]
equals $(1-T^2)^{-N} \mod T^p$.
\end{proof}

As before we have the following corollary.

\begin{corollary}
\label{cor:compactcovering_general}
Let $p = 2m +1$ with $m \geq 0$. Then
\begin{eqnarray}
Q_{\mathrm{pr}_{[0,\sigma]}(A)}  &=&  (1 - T)^N\;Q_{J^p_{\C,\mathrm{pr}_{[0,\sigma]}}(A)}  \mod T^{m+1}.
\end{eqnarray}
\end{corollary}

It is clear from the definition that the complex joins of 
languages in 
$\mathbf{P}_\C$ also belong to 
the complexity class
$\mathbf{P}_\C$. We record this observation
formally in the following proposition.

\begin{proposition}[Polynomial time membership testing]
\label{prop:polytime}
Suppose that the sequence of constructible sets
$(S_n \subset \PP^{k(n)}_\C \times \PP^{\ell(n)}_\C)_{n > 0} \in
\mathbf{P}_\C$,
and 
$\X_n = (X_0:\cdots:X_{k(n)})$ 
$\Y_n = (Y_0:\cdots:Y_{\ell(n)})$ are homogeneous co-ordinates
of $\PP_\C^{k(n)}$ and $\PP_\C^{\ell(n)}$ respectively.
Let $p(n)$ be a non-negative polynomial, and for each $n >0$ let
\[
\mathrm{pr}_{1}: \PP^{k(n)}_\C \times \PP^{\ell(n)}_\C \rightarrow
\PP^{k(n)}_\C 
\]
denote the projection on the first component.

Then, 
\[
\left(J^{p(n)}_{\C,\mathrm{pr}_{1}}(S_n) \subset \PP_\C^{k(n)} \times \PP^{(p(n)+1)(\ell(n)+1)-1}_\C 
\right)_{n > 0} \in {\mathbf{P}}_\C.
\]
\end{proposition}
\begin{proof}
Obvious from the definition of $(J^{p(n)}_{\C,\mathrm{pr}_{1}}(S_n))_{n >0}$.
\end{proof}

\section{Proof of the main theorem}
\label{sec:proof}

We are now in a position to prove Theorem \ref{the:main}.
The proof relies on the following key proposition.

\begin{proposition}
\label{prop:main0}
Let $m(n),k_1(n),\ldots,k_\omega(n)$ be polynomials,  
and let 
$$
\displaylines{
\left(\Phi_n(\X,\Y)\right)_{n>0}
}
$$
be a sequence of multi-homogeneous formulas 
$$ \Phi_n(\X,\Y) \defeq  
(\mathrm{Q}_1 \ZB^1 \in \PP_\C^{k_1})  \cdots 
(\mathrm{Q}_\omega \ZB^\omega \in \PP_\C^{k_\omega})
\phi_n(\X;\Y;\ZB^{1};\cdots;\ZB^{\omega}),$$
having free variables $(\X;\Y) = (X_0,\ldots,X_{k(n)};Y_0,\ldots,Y_{m(n)})$,
with 
\[
\mathrm{Q}_1,\ldots,\mathrm{Q}_\omega \in \{\exists,\forall\}, 
\]
and $\phi_n$ a  multi-homogeneous quantifier-free formula
defining a closed (resp. open) constructible subset 
$$
S_n \subset \PP^k_\C \times \PP^m_\C \times \PP^{k_1}_\C \times \cdots \times
\PP^{k_\omega}_\C.
$$

Suppose also that 
\[
\left(\RR(\phi_n(\X;\Y;\ZB^{1};\cdots;\ZB^{\omega}))\right)_{n > 0}
\in
{\mathbf{P}}_\C.
\]

Then there exist:
\begin{enumerate}
\item 
a sequence of quantifier-free multi-homogeneous formulas 
$$
\displaylines{
\left(
\Theta_n(\X;\V^1;\cdots;\V^{N})
\right)_{n>0},
}
$$
with
$\V^i = (V_0,\ldots,V_{p_i})$, and $N,p_1,\ldots,p_{N}$ 
polynomials in $n$,
such that 
for all $\x \in \PP_\C^k$
$\Theta_n(\x;\V^1;\cdots;\V^{N})$ defines  a 
constructible subset $T_n \subset \PP^{p_1}_\C\times\cdots\times\PP^{p_{N}}_\C$, 
with
\[
\left(T_n\right)_{n >0}\in {\mathbf{P}}_\C;
\] 
\item
polynomial time computable maps
$$
\displaylines{
F_n: \Z[T] \rightarrow \Z[T], 
}
$$
such that
for all $\x \in \PP_\C^k$ 
$$
\displaylines{
Q_{\RR(\Phi_n(\x;\Y))} =  F_n(Q_{\RR(\Theta_n(\x;\V^1;\cdots;\V^N))}).
}
$$
\end{enumerate}
\end{proposition}

The idea behind the proof of Proposition \ref{prop:main0} is to use induction
on the number, $\omega$, of quantifier blocks.
When $\omega=0$, the proposition
is obvious. When $\omega > 0$, then using 
Corollary \ref{cor:compactcovering}, we can construct
a new formula (say $\Phi'_n$)
such that $\Phi'$ has one less block of quantifiers, 
but such that $Q_{\RR(\Phi_n)}$
is easily computable from $Q_{\RR(\Phi_n')}$. One
can then use induction to finish the proof. However, a technical complication
arises due to the fact that in the projective situation (unlike
in the affine case) we cannot immediately replace two adjacent 
blocks of the same quantifier by a single block. This is the 
logical manifestation of the elementary 
fact that the product of two 
projective spaces is not itself a projective space. In order to overcome 
this difficulty and carry through the inductive step properly, we need to prove
a slightly stronger, but technically more 
involved  proposition, which we state next.
Proposition \ref{prop:main0} will be an immediate corollary of this more 
general proposition.

\begin{proposition}
\label{prop:main}
Let $\sigma, \omega \geq 0$ be constants, and
$$
\displaylines{
a_0(n),\alpha_0(n),a_1(n),\alpha_1(n),\ldots,a_{\sigma}(n),\alpha_{\sigma}(n), \cr
k(n),
k_1(n),\ldots,k_\omega(n)
}
$$ 
fixed polynomials in $n$ taking non-negative values for $n \in \mathbb{N}$.

Let $\X = (X_0:\cdots:X_{k(n)})$ denote a block of $k(n)+1$ variables.
For $0 \leq j \leq \sigma$,  let
$\W^j$ denote the tuple of variables 
\[
(\ldots, \W^j_{i_0,\ldots,i_{j}},\ldots), 0 \leq i_0 \leq \alpha_0,
\ldots,0 \leq i_j \leq \alpha_j,
\]
where each
$\W^j_{i_0,\ldots,i_{j}}$ is a block of $a_j(n)+1$ variables.

Let
$$
\displaylines{
\left(\Phi_n(\X;\W^0;\W^1;\ldots;\W^{\sigma})
\right)_{n>0},
}
$$
be a sequence of multi-homogeneous formulas defined by
$$ 
\displaylines{
\Phi_n(\X;\W^0;\W^1;\cdots;\W^{\sigma}) \defeq  \cr
\underset{0 \leq i_0 \leq \alpha_0}{{\Lambda}^0}\cdots 
\underset{0 \leq i_n \leq \alpha_{\sigma}} {{\Lambda}^\sigma}
(\mathrm{Q}_1 \ZB^1 \in \PP_\C^{k_1}) \cdots (\mathrm{Q}_\omega \ZB^\omega \in \PP_\C^{k_\omega})\cr
\phi_{n}(\X;\W^0_{i_0};\W^1_{i_0,i_1};\cdots;\W^{\sigma}_{i_0,\ldots,i_\sigma};\ZB^{1};\cdots;\ZB^{\omega}),
}
$$
with
$$
\displaylines{
{\Lambda}^0,\ldots,{\Lambda}^{\sigma} \in \{\bigvee,\bigwedge\}, 
}
$$
\[
\mathrm{Q}_1,\ldots,\mathrm{Q}_\omega \in \{\exists,\forall\}, 
\]
and each $\phi_{n}$ a multi-homogeneous quantifier-free formula,
multi-homogeneous in the blocks of variables,
$\X,\ZB^1,\ldots,\ZB^\omega$ and 
$(W^j_{i_0,\ldots,i_j,0},\ldots,W^j_{i_0,\ldots,i_j,\alpha_j})$ for $0 \leq j \leq \sigma$.
Suppose also that each 
$\phi_n$ defines a closed (resp. open) constructible set,
and that
\[
\left(\RR(\phi_n)\right)_{n > 0}
\in
{\mathbf{P}}_\C.
\]

Then there exists:
\begin{enumerate}
\item 
a sequence of quantifier-free multi-homogeneous formulas 
$$
\displaylines{
\left(
\Theta_n(\X;\V^1;\cdots;\V^{N})
\right)_{n>0},
}
$$
with
$\V^i = (V_0,\ldots,V_{p_i})$, and $N,p_1,\ldots,p_{N}$ 
polynomials in $n$,
such that $\Theta_n(\x;\V^1;\cdots;\V^{N})$ defines  a 
constructible subset $T_n \subset \PP^{p_1}_\C\times\cdots\times\PP^{p_{N}}_\C$, 
with
\[
\left(T_n\right)_{n >0}\in {\mathbf{P}}_\C;
\] 
\item
polynomial time computable maps
$$
\displaylines{
F_n: \Z[T] \rightarrow \Z[T], 
}
$$
such that 
$$
\displaylines{
Q_{\RR(\Phi_n(\x;\cdot))} =  F_n(Q_{\RR(\Theta_n(\x;\V^1;\cdots;\V^N))}).
}
$$
\end{enumerate}
\end{proposition}

\begin{proof}[Proof of Proposition \ref{prop:main}] 
The proof is by induction on $\omega$.

If $\omega=0$,
we let $\Theta_n = \Phi_n$, 
and 
$F_n$ to be the identity map.
Since there are no quantifiers, for each
$n \geq 0$ the
constructible set
defined by  $\Theta_n$ and  $\Phi_n$ are the same, and thus the
Betti numbers of the sets defined by 
$\Theta_n$ and  $\Phi_n$ are equal.

If $\omega > 0$, we have the following two cases.
\begin{enumerate}
\item
Case 1, $\mathrm{Q}_1 = \exists$: 
First note that $\Phi_n$ defines a constructible subset of 
$\PP_\C^{k(n)} \times U_n$, where 
$$
\displaylines{
U_n = (\PP_\C^{(a_0+1)(\alpha_0 +1)-1})^{m_0}\times  \cdots \times
(\PP_\C^{(a_j+1)(\alpha_j +1)-1})^{m_j} \times \cdots
\times 
(\PP_\C^{(a_\sigma+1)(\alpha_\sigma +1)-1})^{m_\sigma},
}
$$
where for $0 \leq j \leq \sigma$,
$$
\displaylines{
m_j(n) = \prod_{0 \leq i \leq j-1} (\alpha_i(n) +1).
}
$$ 

The formula $\Phi_n$ is equivalent to the following formula:
$$
\displaylines{
\left(
\cdots (\exists \ZB^{1,i_0,\ldots,i_{\sigma}} \in \PP^{k_1}_\C) \cdots
\right)
\bar{\Phi}_n 
}
$$
where
where the blocks of existential quantifiers in the beginning 
are indexed by the tuples 
\[(i_0,\ldots,i_{\sigma}), 0 \leq i_0 \leq \alpha_0(n),
\ldots, 0 \leq i_{\sigma} \leq \alpha_{\sigma}(n).
\]
and the number of such blocks is 
\[
\alpha(n) = \prod_{i=0}^{\sigma} (\alpha_i(n)+1),
\]
and
$$
\displaylines{
\bar{\Phi}_n \defeq 
\underset{0 \leq i_0 \leq \alpha_0}{{\Lambda}^0}\cdots 
\underset{0 \leq i_{\sigma} \leq \alpha_{\sigma}}{{\Lambda}^{\sigma}} \cr 
(\mathrm{Q}_2 \ZB^2 \in \PP_\C^{k_2})  \cdots 
(\mathrm{Q}_\omega \ZB^\omega \in \PP_\C^{k_\omega})
\phi_{n}(\X;\W^0_{i_0};\cdots,\W^{\sigma}_{i_0,\ldots,i_\sigma};
\ZB^{1}
_{i_0,\ldots,i_{\sigma}};\ZB^2\cdots;\ZB^{\omega}),
}
$$

Let
$$
\displaylines{
m(n) = \sum_{j=0}^{\sigma} m_j(n)((a_j(n)+1)(\alpha_j(n) +1)-1). 
}
$$
(Note that $m(n)$ is the (complex) dimension of $U_n$ defined previously.)

Let 
\[
\mathrm{pr}_{1,2}: \PP_\C^{k(n)} \times
U_n \times (\PP_\C^{k_1})^{\alpha(n)}   \rightarrow \PP_\C^{k(n)} \times U_n
\]
denote the projection on the first two components.

Consider the sequence
$$
\displaylines{
\left(
J_{\C,\mathrm{pr}_{1,2}}^{2m(n)+1}(\bar{\Phi}_n)
\right)_{n>0}.
}
$$ 
Note that by \ref{eqn:definitionofjoin2} we have
$$
\displaylines{
J_{\C,\mathrm{pr}_{1,2}}^{2m(n)+1}(\bar{\Phi}_n) = \cr
\underset{0 \leq i_0 \leq \alpha_0}{{\Lambda}^0}
\cdots 
\underset{0 \leq i_{\sigma} \leq \alpha_{\sigma}}{{\Lambda}^{\sigma}}
\;\;
\underset{0 \leq i_{\sigma+1} \leq \alpha_{\sigma+1}}{{\Lambda}^{\sigma+1}} \cr
(\mathrm{Q}_2 \ZB^2 \in \PP_\C^{k_2})  \cdots 
(\mathrm{Q}_\omega \ZB^\omega \in \PP_\C^{k_\omega})\cr
\phi_{n}(\X;\W^0_{i_0};\cdots;\W^{\sigma}_{i_0,\ldots,i_\sigma};
\W^{\sigma+1}_{i_0,\ldots,i_\sigma,i_{\sigma+1}};\ZB^2;\cdots;\ZB^{\omega})
.
}
$$
with ${\Lambda}^{\sigma+1} = \bigwedge$, $\alpha_{\sigma+1} = 2m+1$, and
$\W^{\sigma+1}_{i_0,\ldots,i_{\sigma},i_{\sigma+1}}
= \ZB^{1}_{i_0,\ldots,i_{\sigma},i_{\sigma+1}}$.
We will denote by 
$\W^{\sigma+1}$
the tuple
\[
(\ldots,W^{\sigma+1}_{i_0,\ldots,i_{\sigma},i_{\sigma+1}},\ldots), 0 \leq i_0 \leq \alpha_0, \ldots, 0 \leq i_{\sigma+1} \leq \alpha_{\sigma+1}.
\]

Observe that the  number of quantifiers in the formulas 
$J_{\C,\mathrm{pr}_{1,2}}^{2m(n)+1}(\bar{\Phi}_n),
$ is $\omega-1$.

Moreover, $J_{\C,\mathrm{pr}_{1,2}}^{2m(n)+1}(\bar{\Phi}_n)$
satisfy by Proposition \ref{prop:polytime} 
the required polynomial time hypothesis,
and have the same shape as the formulas $\Phi_n$.
We can thus apply the  induction hypothesis to this
sequence to obtain a sequence
$(\Theta_n)_{n > 0}$, as well as 
a sequence of polynomial time computable 
maps
$(G_n)_{n > 0}$.
By inductive hypothesis we can suppose that
for each $\x \in \PP_\C^{k(n)}$
\[ 
Q_{\RR(J_{\C,\mathrm{pr}_{1,2}}^{2m(n)+1}(\bar{\Phi}_n)(\x;\cdot))} =
G_{n}(Q_{\RR(\Theta_n(\x;\cdot))}).
\]

Using Corollary \ref{cor:compactcovering_general} 
and noticing that the map $\mathrm{pr}_{1,2}$ is either
open or closed and hence a compact covering,

\begin{eqnarray*}
Q_{\RR(\Phi_n(\x;\cdot))} &=& 
(1-T)^{\alpha(n)}Q_{\RR(J_{\C,\mathrm{pr}_{1,2}}^{2m(n)+1}(\bar{\Phi}_n)(\x;\cdot))}
\mod T^{m(n)+1} \\
&=&
(1-T)^{\alpha(n)} G_n(Q_{\RR(\Theta_n(\x;\cdot))}) \mod T^{m(n)+1}.
\end{eqnarray*}

We set 
\[
F_{n}= {\rm Trunc}_{m(n)}\circ M_{(1-T)^{\alpha(n)}}\circ G_{n}
\]
(see Notation \ref{not:poloperators}).
This completes the induction in this case.

\item
Case 2, $\mathrm{Q}_1 = \forall$: 

The formula $\Phi_n$ is equivalent to the following formula:
$$
\displaylines{
\left(\cdots (\forall \ZB^{1,i_0,\ldots,i_{\sigma}} \in \PP^{k_1}_\C) \cdots\right)
\bar{\Phi}_n 
}
$$
where the blocks of universal quantifiers in the beginning 
are indexed by the tuples 
\[(i_0,\ldots,i_{\sigma}), 0 \leq i_0 \leq \alpha_0(n),
\ldots, 0 \leq i_{\sigma} \leq \alpha_{\sigma}(n),
\]
the number of such blocks is 
\[
\alpha(n) = \prod_{i=0}^{\sigma} (\alpha_i(n)+1),
\]
and
$$\displaylines{
\bar{\Phi}_n \defeq
\underset{0 \leq i_0 \leq \alpha_0}{{\Lambda}^0}
\cdots 
\underset{0 \leq i_{\sigma} \leq \alpha_{\sigma}}{{\Lambda}^{\sigma}} \cr 
(\mathrm{Q}_2 \ZB^2 \in \PP_\C^{k_2})  \cdots 
(\mathrm{Q}_\omega \ZB^\omega \in \PP_\C^{k_\omega})
\phi_{n}(\X;\W^0_{i_0};\cdots;\W^{\sigma}_{i_0,\ldots,i_\sigma};
\ZB^{1}_{i_0,\ldots,i_{\sigma}};\ZB^2\cdots;\ZB^{\omega}).
}
$$

Consider the sequence
$$
\displaylines{
\left(
J_{\C,{\mathrm{pr}_{1,2}}}^{2m+1}(\neg \bar{\Phi}_n)
\right)_{n>0}.
}
$$ 

Note that the formula
$$
\displaylines{
J_{\C,\mathrm{pr}_{1,2}}^{2m+1}(\neg \bar{\Phi}_n) = \cr
\underset{0 \leq i_0 \leq \alpha_0}{\bar{{\Lambda}}^0}\cdots 
\underset{0 \leq i_{\sigma} \leq \alpha_{\sigma}} {\bar{{\Lambda}}^{\sigma}}
\underset{0 \leq i_{\sigma+1}\leq \alpha_{\sigma+1}}{{\Lambda}^{\sigma+1}}  \cr
(\bar{Q}_2 \ZB^2 \in \PP_\C^{k_2})  \cdots 
(\bar{Q}_\omega \ZB^\omega \in \PP_\C^{k_\omega})\cr
\neg \phi_{n}(\X;\W^0_{i_0};\cdots;\W^{\sigma}_{i_0,\ldots,i_\sigma};
\W^{\sigma+1}_{i_0,\ldots,i_{\sigma+1}};\ZB^2;\cdots;\ZB^{\omega})
}
$$
with ${\Lambda}^{\sigma+1} = \bigwedge$, $\alpha_{\sigma+1} = 2m+1$, 
$
\displaystyle{ \W^{\sigma+1}_{i_0,\ldots,i_{\sigma+1}}= 
\ZB^{1}_{i_0,\ldots,i_{\sigma},i_{\sigma+1}},
}
$
and
$$
\displaylines{
\bar{{\Lambda}}^i = \bigvee  \mbox{ if }  {\Lambda}^i = \bigwedge, 
\bar{{\Lambda}}^i = \bigwedge \mbox{ if } {\Lambda}^i = \bigvee, \cr
\bar{Q}_i = \exists \mbox{ if } \mathrm{Q}_i = \forall, 
\bar{Q}_i = \forall \mbox{ if } \mathrm{Q}_i = \exists.
}
$$

Observe that the  number of quantifiers in the formulas 
$J_{\C,\mathrm{pr}_{1,2}}^{2m+1}(\neg\bar{\Phi}_n),
$ is $\omega-1$.

Moreover, $J_{\C,\mathrm{pr}_{1,2}}^{2m+1}(\neg\bar{\Phi}_n)$
satisfy by Proposition \ref{prop:polytime} 
the required polynomial time hypothesis,
and have the same shape as the formulas $\Phi_n$.
We can thus apply the  induction hypothesis to this
sequence to obtain a sequence
$(\Theta_n)_{n > 0}$, as well as 
a sequence of polynomial time computable 
maps
$(G_n)_{n > 0}$.
By inductive hypothesis we can suppose that
for each $\x \in \PP_\C^{k(n)}$
\[ 
Q_{\RR(J_{\C,\mathrm{pr}_{1,2}}^{2m+1}(\neg\bar{\Phi}_n)(\x;\cdot))} =
G_{n}(Q_{\RR(\Theta_n(\x;\cdot))}).
\]

Using Corollary \ref{cor:compactcovering_general} 
and noticing that the map $\mathrm{pr}_{1,2}$ is either
open or closed and hence a compact covering,
we have
\begin{eqnarray*}
Q_{\RR(\neg \Phi_n(\x;\cdot))} &=& 
(1-T)^{\alpha(n)}Q_{\RR(J_{\C,\mathrm{pr}_{1,2}}^{2m+1}(\bar{\Phi}_n)(\x;\cdot))}
\mod T^{m(n)+1} \\
&=&
(1-T)^{\alpha(n)} G_n(Q_{\RR(\Theta_n(\x;\cdot))}) \mod T^{m(n)+1}.
\end{eqnarray*}

The sets $K_n = \RR(\Phi_n(\x;\cdot))$ are
constructible and open (resp. closed);
so by
Corollary \ref{cor:alexanderduality_general} (corollary to 
Theorem~\ref{the:alexanderduality_general}), we have

$$
\displaylines{
Q_{K_n}(T) = Q_{U_n}- \Rec_m({\rm Trunc}_m(Q_{U_n - K_n})) .
}
$$

We set 
$F_n$ to be the operator defined by
\begin{equation*} 
F_n(Q) = Q_{U_n} - \Rec_m({\rm Trunc}_m(M_{(1-T)^{\alpha(n)}}(G_n(Q)))).
\end{equation*}
This completes the induction in this case as well.
\end{enumerate}
\end{proof}

\begin{proof}[Proof of Proposition \ref{prop:main0}]
Proposition \ref{prop:main0} is a special case of Proposition
\ref{prop:main} with
$\sigma = 0$, $\alpha_0 = 0$, and $\Y = \W^0$.
\end{proof} 

\begin{proof}[Proof of Theorem \ref{the:main}] 
Follows immediately from Proposition 
\ref{prop:main0} 
in the special case
$m = 0$.
In this case the sequence of formulas $(\Phi_n)_{n>0}$ 
corresponds to a language in the polynomial
hierarchy and for each $n$, $\x = (x_0:\cdots:x_{k(n)}) \in S_n 
\subset\PP_\C^{k(n)}$ 
if and only if
\[
F_{n}(Q_{\RR(\Theta_n(\x;\cdot))})(0) >  0
\]
and 
this last condition can be checked in polynomial time 
using an oracle from
the
class $\#\mathbf{P}^\dagger_\C$.
\end{proof}

\begin{remark}
It is interesting to observe that
in complete analogy with the proof of the classical Toda's 
theorem  the proof of Theorem \ref{the:main} also 
requires just one call to the oracle at the end.
\end{remark}

\begin{proof}[Proof of Theorem \ref{the:main2}]
Follows from the proof of 
Proposition \ref{prop:main0} since the formula
$\Theta_n$ is clearly computable in polynomial time from the given formula
$\Phi_n$ as long as the number of quantifier alternations $\omega$ is 
bounded by a constant.
\end{proof}

\section{Future Directions}
\label{sec:future}
In this section, we sketch
a few directions in which the work presented in this paper 
could be developed further.

\begin{enumerate}
\item
Remove the compactness hypothesis from the main theorem.
\item
The compact fragment of the polynomial hierarchy introduced in this paper,
and especially the class 
$\mathbf{\Sigma}_{\C,1}^c$ (which is the compact fragment of 
$\mathbf{NP}_\C$), is possibly interesting on their
own, and it would be nice to develop a theory of compact reductions and 
compact hardness, and  
have $\mathbf{NP}^c_\C$-complete problems. The compact feasibility
problem discussed in Example \ref{eg:compact} 
is a good candidate for being a 
$\mathbf{NP}^c_\C$-complete problem.
\item
As remarked earlier, one would obtain
a stronger reduction result if one
could prove a Toda-type theorem using only the Euler-Poincar\'e characteristic
instead of the whole Poincar\'e polynomial. This seems to be rather difficult.
An intermediate goal could be to use the 
{\em virtual Poincar\'e polynomial}. The 
virtual Poincar\'e polynomial of a complex variety $X$ is defined by 
\[
\mathcal{P}_X(T) = H_X(-T,-T),
\]
where $H_X(u,v) \in \mathbb{Z}[u,v]$ is the Hodge-Deligne polynomial uniquely
determined by the following properties.
\begin{enumerate}
\item
The map $X\mapsto H_X$ gives an additive and multiplicative invariant from the
Grothendieck ring of equivalence classes of complex varieties to 
$\mathbb{Z}[u,v]$.
\item
$H_X(u,v)$ coincides with $\sum (-1)^{p+q} h^{p,q}(X) u^pv^q$ when $X$ is 
smooth and projective, where $h^{p,q}(X)$ are the Hodge numbers.
\end{enumerate}
Clearly, the virtual Poincar\'e polynomial is additive, and coincides with the
ordinary Poincar\'e polynomial, $P_X$, in the case when $X$ is smooth and 
projective. Thus, one could try to prove the results in this paper using the
virtual Poincar\'e polynomial instead of the Poincar\'e polynomial. 
Unfortunately, the virtual Poincar\'e polynomial is an algebro-geometric, 
rather than  topological invariant, and the topological methods used in this
paper are not sufficient to obtain such a result. 
In particular, Theorem \ref{the:compactcovering} 
does not hold for
the virtual Poincar\'e polynomial except in the trivial case when 
$A = \PP^k_\C \times \PP^\ell_\C$.
\item
Theorem \ref{the:compactcovering} 
can be used to bound the Betti numbers of the images
of complex varieties under regular maps (in conjunction with tight bounds
on the Betti numbers of complex projective varieties due to Katz \cite{Katz}),
instead of first using elimination methods, and then applying the bounds due
to Katz.
A similar method was used in \cite{GVZ04} to obtain bounds on the Betti numbers
of projections of semi-algebraic sets in the real case. One can also treat a
complex variety as a real semi-algebraic set by separating the real and 
imaginary parts, but the  direct complex
method using Theorem \ref{the:compactcovering}  
suggested above should yield better upper bounds on the Betti numbers of 
projections.
\end{enumerate}
\section{Acknowledgments}
The author benefited from discussions with Donu Arapura, 
Andrei Gabrielov, and
James McClure at various stages of this work. The author also thanks an 
anonymous
referee for making several suggestions which helped to improve this paper.

\bibliographystyle{amsplain}
\bibliography{master}
\end{document}